\documentclass[11pt]{amsart}
\usepackage{fullpage}
\usepackage[margin=1in]{geometry}
\usepackage{verbatim}
\usepackage[foot]{amsaddr}
\usepackage{lineno}
\usepackage{amsmath}
\usepackage{bbm}
\usepackage{graphicx}
\usepackage{float}
\usepackage{mathrsfs}
\usepackage[english]{babel}
\usepackage{epstopdf}
\usepackage{algorithmic}
\algsetup{indent=7mm}
\usepackage{amsthm}
\usepackage{amssymb}
\usepackage{listings}
\usepackage{xcolor}
\usepackage{enumerate}
\usepackage{thmtools}
\usepackage{thm-restate}
\usepackage[ruled,vlined]{algorithm2e}
\usepackage{epsfig}
\providecommand{\U}[1]{\protect\rule{.1in}{.1in}}
\usepackage[all]{xy}
\usepackage{caption}
\usepackage{longtable}
\usepackage{mathtools}
\usepackage{subcaption}
\usepackage{tikz-cd}
\usepackage{url}
\usepackage[pdftex,pdfborder={0 0 0},
colorlinks=true,
linkcolor=blue,
citecolor=red,
pagebackref=false,
]{hyperref}

\usepackage{listings}
\usepackage{xcolor}
\usepackage{enumerate}
\usepackage{thmtools}
\usepackage{thm-restate}
\usepackage{epsfig}
\usepackage[all]{xy}
\usepackage{caption}
\usepackage{longtable}
\usepackage{mathtools}
\usepackage{url}
\usepackage{tcolorbox}
\usepackage{multirow}

\newcommand{\R}{\mathbb{R}}

\def\br{\textbf{BR}}
\def\sbr{\textbf{br}}

\def\val{\textbf{val}}

\DeclareMathOperator*{\argmin}{argmin}

\newtheorem{thm}{Theorem}[section]
\newtheorem{prop}[thm]{Proposition}
\newtheorem{lem}[thm]{Lemma}
\newtheorem{cor}[thm]{Corollary}

\theoremstyle{definition}
\newtheorem{ass}{Assumption}

\newtheorem{defn}[thm]{Definition}

\theoremstyle{remark}
\newtheorem{rem}[thm]{Remark}

\numberwithin{equation}{section}

\definecolor{dgreen}{rgb}{0.00,0.49,0.00}
\definecolor{Brown}{rgb}{0.45,0.0,0.05}

\title{Mean field optimization problems: stability results and Lagrangian discretization}

\author{Kang Liu$^{1,2}$ and Laurent Pfeiffer$^1$}
\address{$^1$Universit{\'e} Paris-Saclay, CNRS, CentraleSup{\'e}lec, Inria, Laboratoire des signaux et syst{\`e}mes, 91190, Gif-sur-Yvette, France.}
\address{$^2$Institut Polytechnique de Paris, CNRS, Ecole Polytechnique, CMAP, 91120 Palaiseau, France.}
\email{kang.liu@polytechnique.edu, laurent.pfeiffer@inria.fr}
\date{\today}

\begin{document}

\maketitle

\begin{abstract}
We formulate and investigate a mean field optimization (MFO) problem over a set of probability distributions $\mu$ with a prescribed marginal $m$. The cost function depends on an aggregate term, which is the expectation of $\mu$ with respect to a contribution function. This problem is of particular interest in the context of Lagrangian potential mean field games (MFGs) and their discretization.
We provide a first-order optimality condition and prove strong duality. We investigate stability properties of the MFO problem with respect to the prescribed marginal, from both primal and dual perspectives. In our stability analysis, we propose a method for recovering an approximate solution to an MFO problem with the help of an approximate solution to an MFO with a different marginal $m$, typically an empirical distribution. We combine this method with the stochastic Frank-Wolfe algorithm of \cite{bonnans2022large} to derive a complete resolution method.
\end{abstract}

\medskip


\section{Introduction}

This article is dedicated to a general class of optimization problems involving probability measures with a prescribed marginal $m$. We will refer to them as Mean Field Optimization (MFO) problems. These typically arise in multi-agent optimization problems, for which a mean-field formulation of the problem, involving the probability distribution of the decisions of the agents (rather than an enumeration of them), is not only meaningful but also provides us with convexity properties of great numerical interest.

The first ambition of our work is to provide a general framework for the formulation of such situations. For the sake of clarity, we introduce here the MFO problems investigated in this work. We refer the reader to Sec.\@ \ref{sec:setting} for a complete description of the required assumptions. Let $X$ and $Y$ be two complete and separable metric spaces and let $\mathcal{H}$ be a separable Hilbert space. Let $Z$ be a closed subset of $X\times Y$ and let $m$ be a probability measure on $X$. We consider the following problem, parametrized by $m$:
\begin{equation}\label{pb:primal}\tag{P$_m$}
\inf_{\mu \in \mathcal{P}_{m}(Z)} \  f \left( \int_{Z} g(x,y) d \mu(x,y)  \right),
\end{equation}
where $g\colon Z \to \mathcal{H}$ is a Borel measurable function and $f\colon \mathcal{H}\to \R$ is a convex function.
The admissible set $\mathcal{P}_{m}(Z)$ is the set of all probability measures on $Z$ whose marginal distribution on $X$ is $m$. Our model allows for heterogeneity within the agents, which is modeled by some parameter $x \in X$. The decision variables of the agents are generically denoted by $y \in Y$ and the probability measure $\mu$ represents the distribution of the parameter-decision pairs $(x,y)$ of our agents. The distribution of the parameters of the agents is given by $m$, that is why we impose that $\mu$ has its first marginal equal to $m$. At an abstract level, we can interpret the term $\int_Z g d \mu$ as a common good, obtained by aggregating the contributions of all the agents.

\subsubsection*{Motivation.} Our original interest for MFO problems comes from non-atomic games with a potential structure, for which finding a Nash equilibrium is equivalent to solving an MFO problem. Among these games, we have a special interest for Lagrangian Mean Field Games (MFGs), in which the agents each optimize the trajectory of some dynamical system and are parametrized by their initial condition.
Problem \eqref{pb:primal} also arises in energy management problems, more specifically, in problems involving many small consumption (or production) units, for example electrical cars. For such problems, the common good is the total energy consumption, the parameter $x$ could model any relevant characteristic of the cars (such as their charging capacity) while the variable $y$ describes their charging profiles.
Finally, let us mention that MFO problems find applications in supervised learning, more specifically in the training of neuron networks with one hidden layer: in the mean field approximation of such problems, the variable $\mu$ simply describes the probability distribution of the weights of the neurons \cite{chizat2018global,mei2018mean,chen2023entropic}. For other applications of MFO problems in learning, we refer to \cite{boyd2017alternating,chizat2018global,wang2017vanishing} and the references therein.

\subsubsection*{Numerical approach}
The numerical resolution of Problem \eqref{pb:primal} poses two main difficulties: the numerical manipulation of probability measures and the treatment of the marginal constraint. Let us focus on the first difficulty by supposing momentarily that $m$ is a Dirac measure located at some point $x^*$, so that the problem \eqref{pb:primal} can simply be written as an optimization problem on $\mathcal{P}(Y)$. Unless $Y$ is finite, $\mathcal{P}(Y)$ is an infinite dimensional set. A first approach would consist in discretizing the set $Y$, which would preserve the convexity of the problem. This approach suffers from the curse of dimensionality, since it requires an exponential number of points with respect to the dimension of $Y$. A second approach would consist in representing the probability measure as the empirical mean of a set of $N$ points to be optimized. The clear drawback of this approach is the loss of convexity of the discretized problem; yet we mention that it proves efficient in the context of supervised learning problem \cite{chizat2018global}.
In the general context of MFO problems, the Frank-Wolfe (FW) algorithm is a particularly advantageous algorithm. It produces a finitely supported approximate solution and leverages the convexity of the original problem (without using a coarse discretization of $Y$). More specifically, it generates after $k$ iterations an $\mathcal{O}(1/k)$-optimal solution supported by at most $k$ points.

Coming back to the case of a general marginal, we propose to discretize $m$ with an empirical measure $m_N$ associated with $N$ points and thus to solve:
\begin{equation}\label{pb:primal_dis}\tag{P$_{m_N}$}
\inf_{\mu \in \mathcal{P}_{m_N}(Z)}  f \left( \int_{Z} g d \mu  \right).
\end{equation}
This idea was already proposed in \cite{sarrazin2022lagrangian}, in an MFG context.
By the disintegration theorem, \eqref{pb:primal_dis} is equivalent to a problem involving $N$ probability measures. A direct implementation of the FW algorithm would lead to an approximate solution possibly involving $kN$ points after $k$ iterations of the algorithm. We will see that the Stochastic Frank-Wolfe algorithm, which we introduced and analyzed in \cite{bonnans2022large}, allows to obtain an approximate solution of \eqref{pb:primal_dis} relying on only $N$ support points, to the price of an additional error term of order $\mathcal{O}(1/N)$ in the main convergence result.
Let us mention that the Frank-Wolfe algorithm (also called conditional gradient method) was already applied to potential MFGs, see for example
\cite{geist2022concave,lavigne2022generalized,liu2023mesh}. It can be seen as a generalization of the fictitious play, investigated in particular in
\cite{cardaliaguet2017learning,hadikhanloo2019finite}.

\subsubsection*{Theoretical results} At a theoretical level, we first establish a first-order necessary and sufficient optimality condition for \eqref{pb:primal} and an existence result, both relying on rather standard arguments. Then we perform a stability analysis of the problem with respect to the marginal $m$. It is of course motivated by the need to understand the effect of the discretization of $m$ in the numerical approach described above. We provide a constructive method, which we call \emph{bridging method}. It allows to construct an approximate solution to \eqref{pb:primal}, given an approximate solution to the problem with a different (but close) marginal. This allows to prove that the value of problem \eqref{pb:primal} is Lipschitz continuous with respect to $m$, for the Kantorovich-Rubinstein distance. Finally, we introduce a dual problem to \eqref{pb:primal} and prove that strong duality holds. We prove that the unique solution to the dual problem has a H\"older dependence with respect to $m$.

\subsubsection*{Organization}

In Section \ref{sec:pre}, we present some notations and results in measure theory and set-valued functions, as well as the rigorous description of the data of problem \eqref{pb:primal}. Section \ref{sec:primal} is dedicated to the primal problem: We provide a first-order optimality condition and an existence result. We perform in Section \ref{sec:stable} a stability analysis for the primal problem, based on our bridging method. In Section \ref{sec:dual}, we formulate the dual problem of \eqref{pb:primal}, we prove strong duality, and we prove the stability of the dual solution. We provide our numerical method in Section \ref{sec:algo}. We perform in Section \ref{sec:examples} some numerical simulations for a Lagrangian MFG model taken from \cite{graewe2022maximum} and for a congestion problem.

\section{Preliminaries}\label{sec:pre}

\subsection{Results in measure theory}

A metric space is called a \textit{Polish} space if it is complete and separable.
Let $X$ be a Polish space equipped with a metric $d_X$, and let $\mathcal{X}$ be a $\sigma$-algebra on $X$. The Borel $\sigma$-algebra on $X$ is denoted by $\mathcal{B}^X$. Given any measure $m$ on $\mathcal{X}$, we refer to the triplet $(X, \mathcal{X}, m)$ as a measure space. Measure spaces are said to be complete if for any $A\in \mathcal{X}$ with $m(A)=0$ and for any subset $B$ of $A$, we have $B\in \mathcal{X}$.
We define
\begin{align*}
\mathcal{P}(X) &\coloneqq \big\{ m \text{ is a positive Borel measure on } X, \text{ and } m(X) =1 \big\};\\[0.4em]
\mathcal{P}^1(X) &\coloneqq \left\{ m\in \mathcal{P}(X) \, \Big| \, \exists\, x_0\in X \text{ such that } \int_{X}d_X(x,x_0)dm < +\infty \right\}.
\end{align*}
Let $\delta_x$ denote the Dirac measure at point $x$. We denote by $\mathcal{P}_{\delta}(\Omega)$ the set of finitely supported probability measures, defined by
\begin{equation*}
\mathcal{P}_{\delta}(X)
\coloneqq
\Bigg\{
\sum_{k=1}^K \omega_k \delta_{x_k}
\, \Big| \, K \in \mathbb{N}, \, (\omega_k)_{k=1}^K \in (\R_+)^K, \, (x_k)_{k=1}^K \in X^K, \, \sum_{k=1}^K \omega_k = 1
\Bigg\}.
\end{equation*}
In particular, we call $m\in \mathcal{P}_{\delta}(X)$ an \textit{empirical distribution} if $\lambda_k = 1/K$ for $k=1,\ldots, K$.

The set $\mathcal{P}(X)$ is endowed with the narrow topology. We say that a sequence $(m_n)_{n\geq 1}$ in $\mathcal{P}(Z)$ narrowly converges to some $m\in\mathcal{P}(X)$ if for any bounded and continuous function $F\colon X\to \R$,
\begin{equation*}
\lim_{n\to +\infty} \int_{X} F d m_n = \int_{X} F d m.
\end{equation*}
The space $\mathcal{P}^{1}(X)$ is endowed with the \textit{Kantorovich–Rubinstein Distance},
\begin{equation*}
d_1(m_0,m_1) \coloneqq \sup_{F\in \text{Lip}_1(X) } \int_{\Omega} F  d (m_0- m_1),
\end{equation*}
where Lip$_1(X)$ is the set of all 1-Lipschitz continuous functions on $X$.
For any $m\in \mathcal{P}(X)$, the support of $m$ is defined by
\begin{equation}
\label{eq:support}
\text{supp}(m) \coloneqq \big\{ x\in X \, \mid\, m(V)>0 \text{ for all open set $V$ such that } x \in V \big\}.
\end{equation}

\begin{lem}\label{lm:support}
Let $m\in \mathcal{P}(X)$.
Let $F \colon X \to \R_{+}$ be a Borel measurable function. Assume that
\begin{equation*}
\int_{X} F dm = 0.
\end{equation*}
Then $F = 0$, $m$-a.e. Moreover, if $F^{-1}(\{0\})$ is closed, then \textnormal{supp}$(m) \subseteq F^{-1}(\{0\})$.
\end{lem}

\begin{proof}
The fact that $F=0$, $m$-a.e., is from \cite[Thm.\@ 1.39(a)]{rudin}. Now, let $F^{-1}(\{0\})$ be closed.
Suppose that there exists $x \in \textnormal{supp}(m)$ such that $x \notin F^{-1}(\{0\})$. Since $F^{-1}(\{0\})$ is closed, there exists an open neighborhood $V$ of $x$ such that $F(x) > 0$, for all $x\in V$. By the definition of the support of a probability measure, we have $m(V) > 0$. Therefore, $ \int_{X} F d m \geq \int_{V} F d m > 0$, contradiction.
\end{proof}

\subsection{Results about set-valued functions}

In this subsection, we consider a metric space $X$ equipped with a metric $d_X$, a $\sigma$-algebra $\mathcal{X}$ on $X$, and a measure $m$ on $\mathcal{X}$. Additionally, we fix a Polish space $Y$ with a metric $d_Y$, and we denote the Borel $\sigma$-algebra on $Y$ by $\mathcal{B}^Y$.
We call $F$ a set-valued function from $X$ to $Y$ if $F(x) \subseteq Y$ for all $x\in X$, denoted by $X\rightsquigarrow Y$ for short. The graph of $F$ is defined by
\begin{equation*}
\text{Graph}(F) \coloneqq \left\{
(x,y)\in X\times Y \, \mid \,  y \in F(x)  \right\}.
\end{equation*}
We say that $F$ has closed (non-empty) images, if for any $x\in X$, $F(x)$ is closed (non-empty) in $Y$.

Let us give some definitions concerning regularity properties of set-valued functions, which are from \cite[Def.\@ 1.4.1, Def.\@ 1.4.2, Def.\@ 1.4.5, and Def.\@ 8.1.1]{aubin2009set}.

\begin{defn}\label{def1}
Let $F\colon X\rightsquigarrow Y$ be a set-valued function with non-empty images.
\begin{enumerate}
\item (Lower semi-continuity). The set-valued function $F$ is \textit{lower semi-continuous} at point $x\in X$ if  for any $y\in F(x)$ and any sequence $(x_n \in X)_{n\geq 1}$ converging to $x$, there exists $y_n\in F(x_n)$ converging to $y$. The set-valued function $F$ is said to be lower semi-continuous if it is lower semi-continuous at each point $x\in X$.
\item (Upper semi-continuity). The set-valued function $F$ is \textit{upper semi-continuous} at point $x\in X$ if for any neighborhood $\mathcal{U}$ of $F(x)$, there exists $\eta>0$ such that for any $x'\in B_{X}(x,\eta)$, we have
\begin{equation*}
F(x')\subseteq \mathcal{U}.
\end{equation*}
The set-valued function $F$ is said to be upper semi-continuous if it is upper semi-continuous at each point $x\in X$.
\item (Lipschitz continuity). When $X$ and $Y$ are normed vector spaces, we say that $F$ is $L$-Lipschitz continuous on $X$, for some $L>0$, if for any $x_1, x_2 \in X$,
\begin{equation*}
F(x_1) \subseteq F(x_2) + B_Y(0, Ld_{X}(x_1,x_2)).
\end{equation*}
Here $B_Y(0,r)$ denotes the closed ball in $Y$ centered at $0$ with radius $r>0$.
\item (Measurability). The set-valued function $F$ is \textit{measurable} if the inverse image of any open subset $\mathcal{O}$ of $Y$ is measurable, i.e.,
\begin{equation*}
F^{-1}(\mathcal{O})\coloneqq \left\{ x\in X \,\mid\, F(x) \cap \mathcal{O} \neq  \emptyset \right\} \in \mathcal{X}.
\end{equation*}
\end{enumerate}
\end{defn}

An important property of measurable set-valued functions is the existence of measurable selections.

\begin{thm}[Measurable selection]\label{thm:measurable_selection}
Let $F\colon X\rightsquigarrow Y$ be a measurable set-valued function with non-empty images. Then $F$ has a \textit{measurable selection} $f$, i.e.,
$f\colon X\to Y$ is $(\mathcal{X},\mathcal{B}^Y)$-measurable and $f (x) \in F(x)$ for any $x\in X$.
\end{thm}

\begin{proof}
See \cite[Thm.\@ 8.1.3]{aubin2009set}.
\end{proof}

The following two lemmas will allow us to prove the measurability of some set-valued functions.

\begin{lem}\label{lm:set_measruable_1}
If $F\colon X\rightsquigarrow Y$ is a set-valued function such that $F^{-1}(\mathcal{C})\in \mathcal{X}$ for any closed subset $\mathcal{C}$ of $Y$, then $F$ is measurable.
\end{lem}

\begin{proof}
See \cite[Prop.\@ III.11]{castaing2006convex}.
\end{proof}

\begin{lem}\label{lm:set_measurable_2}
Let $(X, \mathcal{X}, m)$ be a complete measure space, with $m$ a positive measure such that $m(X)=1$. Then any set-valued mapping $F\colon X\rightsquigarrow Y$ is measurable if and only if $\text{Graph}(F)$ belongs to $\mathcal{X}\otimes \mathcal{B}^Y$.
\end{lem}

\begin{proof}
See \cite[Thm.\@ 8.1.4]{aubin2009set}.
\end{proof}

\subsection{Data setting and technical lemmas}\label{sec:setting}

Recall the MFO problem \eqref{pb:primal}. We consider the following setting:
\begin{itemize}
\item Two Polish spaces and their Borel $\sigma$-algebras: $(X, \mathcal{B}^X)$ and $(Y,\mathcal{B}^Y)$.
\item A probability distribution on $X$: $m \in \mathcal{P}(X)$.
\item A set-valued function $F\colon X\rightsquigarrow Y$ with a closed graph and non-empty images. Let
\begin{equation*}
Z \coloneqq \text{Graph}(F), \qquad Z_x \coloneqq F(x), \ \forall x \in X.
\end{equation*}
\item The admissible set of probability measures:
\begin{equation*}
\mathcal{P}_m(Z) \coloneqq \left\{ \mu\in  \mathcal{P}(Z)\, \mid \, \pi_1\# \mu = m  \right\},
\end{equation*}
where $\pi_1\colon Z \to X, \,  (x,y) \mapsto  x$.
\item A separable Hilbert space: $\mathcal{H}$.
\item Two Borel measurable functions: $g\colon Z \to \mathcal{H}$ and $f\colon \mathcal{H}\to \R$.
\end{itemize}

The integral $\int_{Z}g d\mu$ in \eqref{pb:primal} should be interpreted in the Bochner integration sense. We refer to \cite[Appx.\@ E]{cohn2013measure} for Bochner integrable functions.

\begin{lem}\label{lm:Bochner}
If there exists a constant $M>0$ such that $\|g(z)\|\leq M$ for any $z\in Z$, then the function $g$ is Bochner integrable with respect to any $\mu\in \mathcal{P}(Z)$, i.e., $\int_{Z}gd\mu$ exists. Moreover, for any $\lambda\in \mathcal{H}$, we have
\begin{equation*}
\left\langle \lambda, \int_{Z} g d \mu \right\rangle = \int_{Z} \langle \lambda, g\rangle d\mu.
\end{equation*}
As a consequence, for any $\mu_1,\mu_2 \in \mathcal{P}(Z)$, we have
\begin{equation*}
\left\langle \int_{Z} g d\mu_1\, , \, \int_{Z} g d\mu_2\right\rangle = \int_{Z}\int_{Z} \langle g(x),g(y) \rangle d\mu_1(x)d\mu_2(y).
\end{equation*}
\end{lem}

\begin{proof}
As $\mathcal{H}$ is separable, the function $g$ is strongly measurable. Moreover, as the constant function $M$ is Bochner integrable with respect to any $\mu\in\mathcal{P}(Z)$, and $|g(z)|\leq M$ for any $z\in Z$, it follows from \cite[Prop.\@ E.2, Thm.\@ E.6]{cohn2013measure} that $g$ is Bochner integrable with respect to any $\mu\in \mathcal{P}(Z)$. Therefore, we can apply \cite[Prop.\@ E.11]{cohn2013measure} to obtain the first equality of this lemma. The second equality is obtained by applying twice the first one.
\end{proof}

\begin{thm}[Disintegration theorem]\label{thm:disintegration}
For any $\mu \in \mathcal{P}_{m}(Z)$, there exists a family of probability measures $\{\mu_{x \in X} \in \mathcal{P}(Y)\}_{x}$ such that for any Borel measurable function $f \colon Z \to \R_{+}$, we have
\begin{equation*}
\int_{Z} f d\mu  = \int_{X} \int_{Z_{x}} f(x, y) d\mu_{x}(y) dm(x).
\end{equation*}
Moreover, for a.e.\@ $x \in X$, $\mu_x$ is uniquely determined.
\end{thm}

\begin{proof}
See \cite[Thm.\@ 5.3.1]{ambrosio2005gradient}.
\end{proof}

\begin{rem}\label{rem1}
It is not difficult to generalize Theorems \ref{thm:disintegration} to functions $f$ bounded from below, by adding to $f$ a sufficient large positive constant.
\end{rem}

\section{Optimality condition} \label{sec:primal}

\subsection{Assumptions and constants}

To simplify the presentation of the assumptions and the results of the article, we introduce the following (set-valued) functions, parameterized by $\lambda \in \mathcal{H}$:
\begin{itemize}
\item  $g_{\lambda} \colon Z \to \R$ and $u_{\lambda} \colon X \to \R $,
\begin{equation*}
g_{\lambda}(x,y) = \left\langle \lambda \, , \, g(x,y)  \right\rangle, \qquad  u_{\lambda}(x) = \inf_{y \in Z_{x}} g_{\lambda}(x,y);
\end{equation*}
\item $G_{\lambda} \colon X\rightsquigarrow \R $ and $\br_{\lambda} \colon X\rightsquigarrow Y $,
\begin{equation*}
G_{\lambda} (x) = \left\{ g_{\lambda}(x,y) \, \mid\, y \in Z_x \right\}, \qquad \br_{\lambda}(x) = \argmin_{y\in Z_x} g_{\lambda}(x,y).
\end{equation*}
\end{itemize}

\begin{ass}\label{ass1}
The following holds:
\begin{enumerate}
\item
The function $g$ is bounded. The function $f$ is convex and differentiable, and $\nabla f$ is Lipschitz continuous with modulus $L$.
\item  Let $\mathcal{H}_{f} \coloneqq \nabla f(\mathcal{H})$.   Fixing any $\lambda\in \mathcal{H}_f$, we have:
\begin{itemize}
\item  the function $g_{\lambda}$ is lower semi-continuous;
\item  the set-valued function $G_{\lambda}$ is lower semi-continuous;
\item the set-valued function $\br_{\lambda}$ has non-empty images.
\end{itemize}
\end{enumerate}
\end{ass}

Three useful constants below are defined, following Assumption \ref{ass1}:
\begin{equation*}
M \coloneqq \sup_{z\in Z}\|g(z)\|, \qquad D\coloneqq \sup_{z_1,z_2\in Z} \|g(z_1) - g(z_2)\|^2, \qquad C \coloneqq \sup_{\mu\in \mathcal{P}(Z)} \left\|\nabla f\left(\int_{Z}g d\mu \right)\right\|.
\end{equation*}

We present here a lemma following Assumption \ref{ass1}. A similar result for the Lagrangian MFG is presented in \cite[Lem.\@ 3.4]{cannarsa2018existence}.

\begin{lem}\label{lm:close}
Under Assumption \ref{ass1}, for any $\lambda \in \mathcal{H}_f$, the set-valued function $\textnormal{\br}_{\lambda}$ has a closed graph.
\end{lem}

\begin{proof}
Let $x_k\in X$ converge to some $\bar{x} \in X$, and let $y_k \in \br_{\lambda}(x_k)$ converge to some $\bar{y}\in Y$. We have to prove that $\bar{y} \in \br_{\lambda} (\bar{x})$. First, we have $\bar{y} \in Z_{\bar{x}}$, since $Z$ is closed. Fix any $y \in Z_{\bar{x}}$. Since $G_{\lambda}$ is lower semi-continuous, there exists a sequence $(\hat{y}_k)_{k \in \mathbb{N}}$ in $Z_{x_k}$ such that
\begin{equation*}
g_{\lambda}(\bar{x},y) = \lim_{k\to \infty} g_{\lambda}(x_k,\hat{y}_k).
\end{equation*}
By the lower semi-continuity of $g_{\lambda}$, we have
\begin{equation*}
g_{\lambda}(\bar{x},\bar{y}) \leq \liminf_{k\to \infty} g_{\lambda}(x_k, y_k).
\end{equation*}
Since $y_k \in \br_{\lambda}(x_k)$ and $\hat{y}^k \in Z_{x_k}$, we have $ g_{\lambda}(x_k, y_k) \leq  g_{\lambda}(x_k,\hat{y}_k) $ for any $k$. Passing to the limit in this inequality (using the above inequalities), we deduce that $ g_{\lambda}(\bar{x},\bar{y}) \leq   g_{\lambda}(\bar{x},y) $. Thus, $\textnormal{\br}_{\lambda}$ has a closed graph.
\end{proof}

In Section \ref{sec:dual}, we will consider the dual problem of \eqref{pb:primal}. For the analysis of Section \ref{sec:dual}, Assumption \ref{ass1} needs to be strengthened as follows:

\medskip
\noindent\textbf{Assumption A$^{*}$.} Assumption \ref{ass1}(1) holds true and Assumption \ref{ass1}(2) holds true for all $\lambda \in \text{dom}(f^{*})$, where $f^{*}$ is the Fenchel conjugate of $f$.

\begin{rem}
Assumption \textbf{A$^*$} is indeed stronger than Assumption \ref{ass1} since $\mathcal{H}_f = \nabla f(\mathcal{H}) \subseteq \text{dom}(f^{*})$. This inclusion is deduced from Fenchel's relation: $y = \nabla f(x) \Leftrightarrow f^{*}(y) = \langle x, y\rangle - f(x)$.
\end{rem}

\subsection{First-order-optimality condition}

The following lemma plays a key role in proving the first-order optimality condition for \eqref{pb:primal}.

\begin{lem}\label{lm:inf_int}
Let Assumption \ref{ass1} hold true. For any $\lambda\in \mathcal{H}_f$, we have
\begin{equation*}
\inf_{\mu\in \mathcal{P}_{m}(Z)}  \int_{Z} g_{\lambda} d  {\mu}  = \int_{X}  u_{\lambda} d  m.
\end{equation*}
\end{lem}

Here we present a proof of Lemma \ref{lm:inf_int} for the case where $m$ has finite support, that is, $m\in \mathcal{P}_{\delta}(X)$. This particular case provides us with insight into the general proof, and proves beneficial for resolving the discretized problem introduced in Section \ref{sec:algo}.

\begin{proof}[Proof of Lemma \ref{lm:inf_int} when $m\in \mathcal{P}_{\delta}(X)$]
Fix any $\mu \in \mathcal{P}_{m}(Z)$. Since $g$ is bounded over $Z$, the function $g_{\lambda}$ is bounded from below. By Lemma \ref{thm:disintegration} and Remark \ref{rem1}, we have
\begin{equation*}
\int_{Z} g_{\lambda} d  {\mu} = \int_{X} \int_{Z_x} g_{\lambda} (x,y) d \mu_x(y) dm(x) \geq \int_{X} u_{\lambda} dm,
\end{equation*}
where the second inequality follows from the definition of $u_{\lambda}$.

Let us prove the converse inequality. Let us fix
$m\in \mathcal{P}_{\delta}(X)$. Let $K \in \mathbb
N$, let $(x_k)_{k=1,\ldots,K} \in X^k$ and let $(\omega_k)_{k=1,\ldots,K} \in \R_+^K$ be such that $\sum_{k=1}^K \omega_k =1$ and $m = \sum_{k=1}^K \omega_k \delta_{x_k}$. For any $k=1,\ldots,K$, let $y_k \in \br_{\lambda}(x_k)$. Let us define
$\tilde{\mu}
= \sum_{k=1}^K \omega_k \delta_{(x_k,y_k)}.
$
Clearly $\tilde{\mu} \in \mathcal{P}_{m}(Z)$. Moreover,
\begin{equation*}
\int_{Z} g_\lambda d \tilde{\mu}
= \sum_{k=1}^K \omega_k
g_\lambda(x_k,y_k)
= \sum_{k=1}^K \omega_k u_\lambda(x_k)
= \int_X u_\lambda dm.
\end{equation*}
The conclusion follows, moreover, $\tilde{\mu}$ minimizes $\int_{Z} g_\lambda d\mu$ over $\mathcal{P}_m(Z)$.
\end{proof}

In the general case, one has to find a measurable selection of $\br_{\lambda}$, which requires us to prove the measurability of $\br_{\lambda}$, which cannot be done in a direct fashion. The complete proof is given in Appendix \ref{Appx:A}.

\begin{thm}[First-order optimality condition]\label{thm:first}
Let Assumption \ref{ass1}(1) hold true. Let $\bar{\mu} \in \mathcal{P}_{m}(Z)$ and $ \bar{\lambda} = \nabla f \left( \int_{Z} g  d \bar{\mu}\right)$.  Consider the following three assertions:
\begin{enumerate}
\item The measure $\bar{\mu}$ is a solution of problem \eqref{pb:primal};
\item $\int_{Z} g_{\bar{\lambda}} d \bar{\mu} = \inf_{\mu \in \mathcal{P}_m(Z)} \int_{Z} g_{\bar{\lambda}} d{\mu} $;
\item $\textnormal{supp}(\bar{\mu}_x) \subseteq \textnormal{\br}_{\bar{\lambda}}(x)$, $ m\textnormal{-}$a.e., where $\bar{\mu}_x$ is defined by the disintegration theorem.
\end{enumerate}
Then, assertions (1) and (2) are equivalent. Moreover, under Assumption \ref{ass1}(2), assertions (1), (2), and (3) are equivalent.
\end{thm}

\begin{proof}
\textbf{Step 1}. (Equivalence between $(1)$ and $(2)$). We first prove that $(1)\Rightarrow (2)$. Suppose that $\bar{\mu}$ is a solution of problem \eqref{pb:primal}. Take an arbitrary $\mu\in \mathcal{P}_{m}(Z)$. Then, for any $\alpha\in[0,1]$, we have
\begin{equation*}
\begin{split}
f \left( \int_{Z} g d \bar{\mu} \right) &\leq {} f \left( \int_{Z} g  d (\bar{\mu} + \alpha(\mu - \bar{\mu}) ) \right)\\
& \leq {} f \left( \int_{Z} g d \bar{\mu} \right) + \alpha \left\langle \bar{\lambda}\, , \,  \int_{Z} g d (\mu - \bar{\mu} ) \right\rangle + \frac{\alpha^2 LD}{2},
\end{split}
\end{equation*}
where the second inequality follows from the Lipschitz-continuity of $\nabla f$ and the definition of $D$. Therefore
\begin{equation*}
0 \leq \left\langle \bar{\lambda}\, , \,  \int_{Z} g d (\mu - \bar{\mu} ) \right\rangle + \frac{\alpha LD}{2}
\end{equation*}
Let $\alpha$ go to $0$. We obtain that
\begin{equation}\label{eq:first_order_1}
\left\langle \bar{\lambda}\, , \,  \int_{Z} g d  \bar{\mu}  \right\rangle = \inf_{\mu\in \mathcal{P}_{m}(Z)}  \left\langle \bar{\lambda}\, , \,  \int_{Z} g d  {\mu}  \right\rangle.
\end{equation}
This implies $(2)$ by the definition of $g_{\bar{\lambda}}$.

We now prove $(2)\Rightarrow (1)$. Let $(2)$ hold true. We obtain \eqref{eq:first_order_1} by the definition of $g_{\bar{\lambda}}$. The convexity of $f$ implies that for any $\mu \in \mathcal{P}_{m}(Z)$,
\begin{equation*}
f \left( \int_{Z} g d {\mu} \right) \geq f \left( \int_{Z} g d \bar{\mu} \right) + \left\langle \bar{\lambda} \, , \int_{Z} g d {\mu} - \int_{Z} g d \bar{\mu} \right\rangle \geq f \left( \int_{Z} g d \bar{\mu}\right).
\end{equation*}
Therefore, $\bar{\mu}$ is a solution of problem \eqref{pb:primal}.

\textbf{Step 2}. (Equivalence between $(2)$ and $(3)$).
By Theorem \ref{thm:disintegration}, we have
\begin{equation*}
\int_{Z} g_{\bar{\lambda}} d\bar{\mu} =
\int_{X}\int_{Z_x} g_{\bar{\lambda}}(x,y) d\bar{\mu}_x (y) dm(x).
\end{equation*}
By Lemma \ref{lm:inf_int}, we have
\begin{equation*}
\inf_{\mu\in \mathcal{P}_{m}(Z)}  \int_{Z} g_{\bar{\lambda}} d  {\mu}  = \int_{X}  u_{\bar{\lambda}} d  m.
\end{equation*}
Therefore, assertion $(2)$ is equivalent to
\begin{equation}\label{eq:first_order_2}
\int_{X}\int_{Z_x} g_{\bar{\lambda}}(x,y) d\bar{\mu}_x (y) dm(x)= \int_{X} u_{\bar{\lambda}} dm.
\end{equation}

Let $(3)$ hold true. It follows that $ \int_{Z_x} g_{\bar{\lambda}}(x,y) d\bar{\mu}_x (y)=  u_{\bar{\lambda}}(x)$, $ m $-a.e., which implies \eqref{eq:first_order_2}.

Let $(2)$ hold true. We obtain \eqref{eq:first_order_2}.
The function $x \mapsto \big( \int_{Z_x} g_{\bar{\lambda}}(x,y) d\bar{\mu}_x (y) \big) - u_{\bar{\lambda}}(x)$ is nonnegative, for a.e.\@ $x \in X$, by the definition of $u_{\bar{\lambda}}$. By \eqref{eq:first_order_2}, its integral is null, thus, as a consequence of Lemma \ref{lm:support}, we have
\begin{equation} \label{eq:support_relation}
\int_{Z_x} g_{\bar{\lambda}}(x,y) d\bar{\mu}_x (y)= u_{\bar{\lambda}}(x) = \inf_{y\in Z_x} g_{\bar{\lambda}}(x,y), \qquad m\text{-a.e.}
\end{equation}
Fix $x \in X$ such that equality holds in \eqref{eq:support_relation}. Consider the map $y \in Z_x \mapsto g_{\bar{\lambda}}(x,y) - u_{\bar{\lambda}}(x)$. It is nonnegative, with a null integral, and $\br_{\bar{\lambda}}(x)$ is non-empty and closed. Then assertion (3) follows with Lemma \ref{lm:support}.
\end{proof}

\begin{cor}\label{cor:fixed_point}
Under Assumption \ref{ass1}, $\bar{\mu}$ is a solution of \eqref{pb:primal} if and only if the following equilibrium equation is satisfied:
\begin{equation}\label{pb:fix}
\begin{cases}
\ \bar{\lambda} = \nabla f \left(  \int_{Z} g d\bar{\mu} \right),\\[0,6em]
\ \textnormal{supp}(\bar{\mu}_x) \subseteq \textnormal{\br}_{\bar{\lambda}}(x), \quad $m$\text{-a.e.}
\end{cases}
\end{equation}
\end{cor}

\begin{proof}
This is a consequence of Theorem \ref{thm:first}.
\end{proof}

The conditions in \eqref{pb:fix} can be interpreted as the conditions for a Nash equilibrium in an non-atomic game, in which the agents through the variable $\bar{\gamma}$. The relation $\bar{\lambda}= \nabla f(\int_Z g d \bar{\mu})$ shows how $\bar{\lambda}$ results from the collective behavior of the agents, while the relation $\text{supp}(\bar{\mu}_x) \subseteq \textnormal{\br}_{\bar{\lambda}}(x)$ shows that the agents behave optimally, for some criterion that depends on $\bar{\lambda}$.
We will discuss some more concrete examples in Section \ref{sec:examples}.


\subsection{Existence of a solution under tightness assumptions}

We denote by $\val\eqref{pb:primal}$ the value of problem \eqref{pb:primal}. We can easily deduce from Assumption \ref{ass1} that $\val\eqref{pb:primal} > -\infty$. The following proposition demonstrates the existence of a solution to problem \eqref{pb:primal} under some additional assumptions.

\begin{prop}[Existence]\label{prop:existence}
Let Assumption \ref{ass1} hold true.
Let $(\mu_n)_{n\geq 1}$ be a minimizing sequence for problem \eqref{pb:primal}. Suppose that $\{\mu_n\}_{n\geq 1}$ is tight in $\mathcal{P}(Z)$, i.e.
for any $\epsilon >0$, there exists a compact subset $K_{\epsilon}$ of $Z$ such that
\begin{equation*}
\mu_n(K_{\epsilon}) \geq 1 - \epsilon, \qquad \forall n\geq 1.
\end{equation*}
Then every accumulation point of $\{\mu_n\}_{n \geq 1}$ for the narrow topology (there exists at least one) is a solution of \eqref{pb:primal}.
\end{prop}

\begin{proof}
By \textit{Prokhorov's} theorem \cite[p.\@ 43]{villani2009}, the set $\{\mu_n\}_{n\geq 1}$ is relatively compact with respect to the narrow topology. Without loss of generality, suppose that $\mu_n$ narrowly converges to some $\bar{\mu} \in \mathcal{P}(Z)$.
The set $\mathcal{P}_m(Z)$ is closed with respect to narrow topology by \cite[Prop.\@ 2.4]{santambrogio2021cucker}. This implies that
$\bar{\mu} \in \mathcal{P}_m(Z)$. Let $\bar{\lambda} = \nabla f(\int_{Z} g d\bar{\mu})$. Since $f$ is convex, we have
\begin{equation}\label{eq:quad}
f\left( \int_Z g d\mu_n \right) \geq f\left( \int_Z g d\bar{\mu} \right) + \int_{Z} g_{\bar{\lambda}} d (\mu_n - \bar{\mu}).
\end{equation}
Since $g_{\bar{\lambda}}\colon Z \to \R$ is lower semi-continuous and bounded from below by Assumption \ref{ass1}, we deduce the following inequality from \cite[Lem.\@ 4.3]{villani2009}:
\begin{equation*}
\liminf_{n\to +\infty}\int_{Z} g_{\bar{\lambda}} d (\mu_n - \bar{\mu}) \geq 0.
\end{equation*}
In inequality \eqref{eq:quad}, letting $n$ go to infinity, by the definition of $\mu_n$, we have
\begin{equation*}
\val\eqref{pb:primal} = \liminf_{n\to +\infty} f\left( \int_Z g d\mu_n \right) \geq f\left( \int_Z g d \bar{\mu} \right) \geq \val\eqref{pb:primal}.
\end{equation*}
Therefore, $\bar{\mu}$ is a solution of problem \eqref{pb:primal}.
\end{proof}


\section{Stability analysis and bridging method} \label{sec:stable}

In this section, we study the stability of the primal problem \eqref{pb:primal} with respect to its parameter $m$.
We need the following assumptions (recall the data setting introduced in Sec.\@ \ref{sec:setting}).

\begin{ass}\label{ass2} The following holds:
\begin{enumerate}
\item The space $X$ is a closed subset of a separable Banach space;
\item The function $g\colon Z\to \mathcal{H}$ is continuous;
\item The set $Z_x$ is compact for any $x\in X$ and the set-valued function $F\colon X \rightsquigarrow Y$ is upper semi-continuous;
\item There exists $L_g \geq 0$ such that the set-valued function
\begin{equation} \label{eq:def_Z}
\mathcal{Z}\colon X \rightsquigarrow \mathcal{H},\
x\mapsto \{g(x,y) \mid y\in Z_x\}
\end{equation}
is $L_g$-Lipschitz on $X$.
\end{enumerate}
\end{ass}

Let $m_0$ and $m_1$ lie in $\mathcal{P}(X)$. We consider the following two instances of \eqref{pb:primal} with $m= m_0$ and $m=m_1$ respectively:
\begin{align}
\inf_{\mu \in \mathcal{P}_{m_0}(Z)}  f \left( \int_{Z} g d \mu  \right); \label{pb:primal_0}\tag{P$_{m_0}$}\\
\inf_{\mu \in \mathcal{P}_{m_1}(Z)}  f \left( \int_{Z} g d \mu  \right). \label{pb:primal_1}\tag{P$_{m_1}$}
\end{align}

Suppose that we have an (approximate) solution of problem \eqref{pb:primal_0}, denoted by $\bar{\mu}_{0}$. Our goal is to propose a feasible approach for recovering an approximate solution of problem \eqref{pb:primal_1} from $\bar{\mu}_{0}$ and to study the performance of this approximation. We call it the bridging method. It relies on $\bar{\mu}_0$ and the solution of the optimal transport problem \eqref{pb:OT1} stated later. To introduce \eqref{pb:OT1}, we need to define some projection operators

Recall that $\pi_1\colon Z \to X$, $(x,y)\mapsto x$, and $\pi_2\colon Z \to X$, $(x,y)\mapsto y$. The other projection operators used in this subsection are defined as:
\begin{align*}
&\tilde{\pi}_1: X\times X \to X,\, (x,x')\mapsto x, \quad  \tilde{\pi}_2: X\times X \to X, \, (x,x')\mapsto x', \quad  \pi_3 \colon Z\times X \to X, (x,y,x')\mapsto x',
\\
&\pi_{12} \colon Z\times X \to Z,\,  (x,y,x')\mapsto (x,y),  \qquad \pi_{13} \colon Z\times X \to Z,\,  (x,y,x')\mapsto (x,x').
\end{align*}
It directly follows from the above definitions that
\begin{align*}
\tilde{\pi}_1\circ \pi_{13} = \pi_1\circ \pi_{12} \quad
\text{and} \quad
\tilde{\pi}_2\circ\pi_{13} = \pi_3.
\end{align*}

Now we consider the following optimal transport problem:
\begin{equation}\label{pb:OT1}\tag{OT1}
\inf_{\rho \in \Pi(m_0,m_1)}  \int_{X \times X} d_{X}(x,x') d\rho(x,x'),
\end{equation}
where $\Pi(m_0,m_1) = \left\{ \rho \in \mathcal{P}(X \times X)\, \mid \, \tilde{\pi}_{1} \# \rho = m_0, \, \tilde{\pi}_2 \# \rho = m_1  \right\}$.
It follows from \cite[Rem.\@ 6.5]{villani2009} that if $m_0$ and $m_1$ lie in $\mathcal{P}^1(X)$, then $d_{1}(m_0,m_1) = \val\eqref{pb:OT1}$.

The following particular example will provide an intuitive understanding of our bridging method.

\textbf{A particular case}. Let us assume that the distributions $m_0$, $m_1$, and $\bar{\mu}_0$ are empirical distributions with supports of size $N$, i.e., there exists $(x_i)_{i=1}^N, (\tilde{x}_i)_{i=1}^N \in X^N$ and $(y_i)_{i=1}^N\in \prod_{i=1}^N Z_{x_i}$ such that
\begin{equation}\label{eq:empirical}
m_0 = \frac{1}{N}\sum_{i=1}^N \delta_{x_i}, \qquad m_1 = \frac{1}{N}\sum_{i=1}^N \delta_{\tilde{x}_i}, \qquad
\bar{\mu}_0 = \frac{1}{N}\sum_{i=1}^N \delta_{(x_i,y_i)}.
\end{equation}

\begin{lem}\label{lm:rho}
Let $m_0$ and $m_1$ be defined by \eqref{eq:empirical}. Then problem \eqref{pb:OT1} has a solution
\begin{equation}\label{eq:empirical1}
\rho = \frac{1}{N}\sum_{i=1}^N \delta_{(x_i,x_i')},
\end{equation}
where $\{x_1', \ldots, x_N'\}$ is a permutation of $\{\tilde{x}_1,\ldots, \tilde{x}_N\}$.
\end{lem}

\begin{proof}
This is a consequence of \cite[Prop.\@ 2.1]{peyre2019computational}.
\end{proof}

Let $\rho$ be given by Lemma \ref{lm:rho}. By Assumption \ref{ass2}(4), for any $i$, there exists $y_i'\in Z_{x_i'}$ such that
\begin{equation}\label{eq:g}
\|g(x_i',y_i') - g(x_i,y_i)\|\leq L_g d_X(x_i,x_i').
\end{equation}
In our bridging method, each $x_i$ is transformed to $x_i'$ while simultaneously moving $y_i$ to the point $y_i' \in Z_{x_i'}$ for $i=1,\ldots,N$. This can be expressed as follows:
\begin{equation} \label{eq:recovery_easy}
\bar{\mu}_0 = \frac{1}{N}\sum_{i=1}^N \delta_{(x_i,y_i)} \longrightarrow \mu_1 = \frac{1}{N}\sum_{i=1}^N \delta_{(x'_i,y'_i)}.
\end{equation}

To provide a clearer formula of the construction of $\mu_1$, we introduce the empirical distribution $\nu_N \in \mathcal{P}(Z\times X)$ and the mapping $s_N \colon \{(x_i,y_i,x_i')\}_{i=1}^N \to \{(x_i,y_i')\}_{i=1}^N$, defined as:
\begin{align*}
&\nu_N = \frac{1}{N}\sum_{i=1}^N \delta_{(x_i,y_i,x_i')}, \\
&s_N(x_i,y_i,x_i') = (x_i',y_i'),\quad \forall i=1,2,\ldots, N.
\end{align*}
It can be observed that $\pi_{12}\#\nu_N = \bar{\mu}_0$ and $\pi_3\#\nu_N = m_1$. Furthermore, we will demonstrate later in Lemma \ref{lm:OT} that $\nu_N$ is a solution of another optimal transport problem \eqref{pb:OT}.
Then the approximate solution $\mu_1$ of problem \eqref{pb:primal_1} can be written as:
\begin{equation*}
\mu_1 = s_N\# \nu_N.
\end{equation*}
The distribution $\mu_1$ belongs to $\mathcal{P}_{m_1}(Z)$, and furthermore,
\begin{equation*}
\begin{split}
\left\| \int_{Z}gd\bar{\mu}_0 - \int_{Z} g d\mu_1 \right\| & ={}  \left\| \frac{1}{N}\sum_{i=1}^N \big(g(x_i,y_i) - g(x_i',y_i') \big)\right\|
\\
& \leq {} \frac{L_g}{N}\sum_{i=1}^N \|g(x_i,y_i) - g(x_i',y_i')\| \\
&\leq {} \frac{L_g}{N}\sum_{i=1}^N d_X(x_i,x_i') = L_g d_1(m_0,m_1),
\end{split}
\end{equation*}
where the second line follows from the triangle inequality and the third line follows from \eqref{eq:g} and Lemma \ref{lm:rho}. The above inequality demonstrates that the distance between the aggregates associated with $\bar{\mu}_0$ and $\mu_1$ is controlled by the $d_1$-distance of $m_0$ and $m_1$.

\textbf{The general case}. To investigate the stability and present the bridging method in the general case, we draw inspiration from the constructions of $\nu_N$ and $s_N$ in the previous particular case and introduce the following:
\begin{itemize}
\item the auxiliary optimal transport problem:
\begin{equation}\label{pb:OT}\tag{OT2}
\inf_{\nu \in \Pi(\bar{\mu}_0, m_1)} \int_{Z\times X} d_{X}(x,x') d\nu(x,y,x'),
\end{equation}
where $\Pi(\bar{\mu}_0, m_1) \coloneqq \left\{ \nu \in \mathcal{P}(Z \times X)\, \mid \, \pi_{12} \# \nu = \bar{\mu}_0, \, \pi_3 \# \nu = m_1  \right\}$;
\item the set-valued function $S \colon Z\times X \rightsquigarrow  Z$,
\begin{equation}\label{eq:S}
S(x,y,x') = \left\{ (x', y') \in Z \, \mid \, \|g(x', y') -  g(x,y)\|
\leq L_g  d_{X}(x, x') \right\}.
\end{equation}
\end{itemize}

Note that problems \eqref{pb:OT} and \eqref{pb:OT1} are similar in so far as the integrand of the cost function is the same, moreover, the second marginal of $\nu$ in \eqref{pb:OT} (resp.\@ $\rho$ in \eqref{pb:OT1}) must be equal to $m_1$.
The following lemma shows the equivalence between problems \eqref{pb:OT1} and \eqref{pb:OT}. We will see that the solution of \eqref{pb:OT} will play the role of $\nu_N$ in the particular case mentioned earlier.
\begin{lem}\label{lm:OT}
If $m_0$ and $m_1$ lie in $\mathcal{P}^1(X)$, then both problems \eqref{pb:OT1} and \eqref{pb:OT} have solutions, moreover,
\begin{equation*}
\textnormal{\val} \eqref{pb:OT1} =  \textnormal{\val} \eqref{pb:OT} = d_1(m_0,m_1).
\end{equation*}
\end{lem}

\begin{proof}
Since $m_0,m_1 \in \mathcal{P}^1(X)$, by \cite[Rem.\@ 6.5]{villani2009}, we have $d_{1}(m_0,m_1) = \val\eqref{pb:OT1}$.
The existence of solutions of problems \eqref{pb:OT} and \eqref{pb:OT1} is from \cite[Thm.\@ 4.1]{villani2009}.

Let $\nu$ be a solution to \eqref{pb:OT} and let $\rho = \pi_{13}\#\nu$, which is clearly an element of $\mathcal{P}(X\times X)$. By the basic properties of push-forward measures, we have that
$
\tilde{\pi}_1\# \rho = \tilde{\pi}_1\# (\pi_{13}\# \nu) = (\tilde{\pi}_1\circ \pi_{13})\# \nu.
$
Using the relation $\tilde{\pi}_{1}\circ \pi_{13} = \pi_1 \circ \pi_{12}$, we obtain that $ (\tilde{\pi}_1\circ \pi_{13})\# \rho = (\pi_1\circ \pi_{12})\# \rho = \pi_1\# (\pi_{12}\# \nu) = \pi_1\# \bar{\mu}_0 = m_0 $.
It follows that $\tilde{\pi}_1\# \rho = m_0$.
By similar arguments, we deduce that $\tilde{\pi}_2\# \rho = m_1$ from the relation $\tilde{\pi}_2\circ \pi_{13} = \pi_3$. Therefore, $\mu\in \Pi(m_0,m_1)$, moreover,
\begin{equation*}
\int_{X \times X} d_{X}(x,x') d \pi_{13}\#\nu(x,x') = \int_{Z \times X} d_{X}( \pi_{13}(x,y,x')) d \nu(x,y,x') = \val\eqref{pb:OT}.
\end{equation*}
It follows that $ d_1(m_0,m_1) \leq \val\eqref{pb:OT}$.

On the other hand, let $\rho$ be a solution of \eqref{pb:OT1}. Since $\bar{\mu}_0$ and $\rho$ have the same marginal distribution $m_0$ with respect to their first variable, by the Gluing lemma \cite[p.\@ 11]{villani2009}, there exists a probability measure $\nu\in \mathcal{P}(X\times Y \times X)$ such that
\begin{equation*}
\pi_{12} \# \nu = \bar{\mu}_0, \qquad  \pi_{13}\# \nu = \rho.
\end{equation*}
Since $\bar{\mu}_0 \in \mathcal{P}(Z)$, we have $\nu \in \mathcal{P}(Z\times X)$. From the relation $\pi_3 = \tilde{\pi}_2 \circ \pi_{13} $, we deduce that $ \pi_3 \# \nu = \tilde{\pi}_2\# (\pi_{13}\# \nu) = \tilde{\pi}_2\# \rho = m_1$. Thus, $\nu\in \Pi(\bar{\mu}_0, m_1)$, moreover,
\begin{equation*}
\begin{split}
\int_{Z \times X} d_{X}(x,x') d \nu(x, y , x') =  \int_{Z \times X} d_{X}( \pi_{13}(x,y,x')) d \nu(x,y,x')
= \int_{X \times X} d_{X}( x,x') d \pi_{13}\#\nu(x,x')
\\
=\int_{X \times X} d_{X}( x,x') d \rho(x,x') = d_{1}(m_0,m_1).
\end{split}
\end{equation*}
It follows that $ d_1(m_0,m_1) \geq \val\eqref{pb:OT}$.
\end{proof}

The following two lemmas demonstrate that the set-valued function $S$ has a measurable selection, which will be denoted by $s$.
We will see that the function $s$ will fulfill the role of $s_N$ in the particular case discussed earlier.
\begin{lem}\label{lm:usc}
Under Assumption \ref{ass2}(3), let $(x_n )_{n\geq 1}$ be a sequence in $X$ converging to some $x_0\in X$. Then any sequence $(y_n\in Z_{x_n})_{n\geq 1}$ has a convergent sub-sequence with its limit in $Z_{x_0}$.
\end{lem}

\begin{proof}
Since $F\colon X \rightsquigarrow Y$ is upper semi-continuous, for any $k\geq 1$, there exists $\eta_k >0$ such that for any $x\in B_X(x_0,\eta_k)$, we have $Z_{x} \subseteq B_Y(Z_{x_0}, 1/k) \coloneqq \cup_{y \in Z_{x_0}} B_Y(y, 1/k)$. For $k=1$, there exists $\varphi(1)\in \mathbb{N}_{+}$ such that $Z_{x_{\varphi(1)}} \subseteq B_Y(Z_{x_0}, 1)$, i.e., there exists $\bar{y}_1 \in Z_{x_0}$ such that
\begin{equation*}
d_{Y} (y_{\varphi(1)}, \bar{y}_1) \leq 1.
\end{equation*}
Assume now we have $\varphi(k)\in \mathbb{N}_{+}$ and $\bar{y}_k \in Z_{x_0}$  for $k=1,\ldots, K$ such that
\begin{equation*}
d_{Y} (y_{\varphi(k)}, \bar{y}_k) \leq \frac{1}{k}, \qquad k=1,\ldots, K.
\end{equation*}
Since $x_n\to x_0$, we can find $ \varphi(K+1)> \varphi(K)$ such that $x_{\varphi(K+1)} \in B_{X}(x_0,\eta_{K+1})$. As a consequence,
there exists $\bar{y}_{K+1} \in Z_{x_0}$ such that
\begin{equation*}
d_{Y} (y_{\varphi(K+1)}, \bar{y}_{K+1}) \leq \frac{1}{K+1}.
\end{equation*}
Since $Z_{x_0}$ is compact, $(\bar{y}_{k})_{k\geq 1}$ has a convergent sub-sequence $(\bar{y}_{\phi(k)})_{k\geq 1}$ with a limit $\bar{y}\in Z_{x_0}$. By the triangle inequality,
\begin{equation*}
d_Y(y_{\phi(\varphi(k))}, \bar{y}) \leq d_Y(y_{\phi(\varphi(k))}, \bar{y}_{\phi(\varphi(k))}) + d_Y(\bar{y}_{\phi(\varphi(k))}, \bar{y}) \leq \frac{1}{\phi(\varphi(k))} + d_Y(\bar{y}_{\phi(\varphi(k))}, \bar{y}).
\end{equation*}
Since $\phi$ and $\varphi$ are strictly increasing functions going to $+\infty$, we have $ \lim_{k\to +\infty}d_Y(y_{\phi(\varphi(k))}, \bar{y}) = 0$.
Therefore, $( y_{\phi(\varphi(k))})_{k\geq 1}$ is a convergent sub-sequence of $(y_n)_{n\geq 1}$ with its limit $\bar{y}\in Z_{x_0}$.
\end{proof}

\begin{lem} \label{lm:S1}
Under Assumption \ref{ass2}, the set-valued function $S$ has a Borel measurable selection function $s \colon Z\times X \to  Z$. Furthermore,  we have $\|g(s(x,y,x')) -  g(x,y)\| \leq L_g  d_{X}(x, x')$.
\end{lem}

\begin{proof}
We will apply Theorem \ref{thm:measurable_selection} and Lemma \ref{lm:set_measruable_1} to prove the result. The images of $S$ are non-empty since the set-valued mapping $\mathcal{Z}$ (defined in \eqref{eq:def_Z}) is supposed to be $L_g$-Lipschitz.
Let us first verify that $S$ has non-empty closed images. Fix any $(x,y,x')\in Z\times X$, and assume that $(x',z_n)\in S(x,y,x')$ converges to some $(x',z)\in Z$.
It suffices to prove that $(x',z)\in  S(x,y,x')$, i.e., $ \|g(x', z) -  g(x,y)\|
\leq L_g  d_{X}(x, x')$. This is true since $g$ is continuous and $z_n\to z$.

Then, let us show that $S^{-1}(\mathcal{C})$ is closed for any closed subset $\mathcal{C}$ in $Z$.
By \eqref{eq:S}, we have
\begin{align*}
S^{-1}(\mathcal{C})  = & \left\{ (x,y,x')\in Z\times X \, \mid \, S(x,y,x')\cap \mathcal{C} \neq \emptyset \right\} \\
=  & \left\{ (x,y,x')\in Z\times X \, \mid \, \exists\, y' \in Z_{x'} \text{ such that } \begin{cases}
(x', y')\in \mathcal{C}, \\
\|g(x', y') -  g(x,y)\|
\leq L_g  d_{X}(x, x')
\end{cases}
\right\}.
\end{align*}
If $S^{-1}(\mathcal{C})= \emptyset$, then the conclusion is obvious. Assume that $S^{-1}(\mathcal{C}) \neq \emptyset$ and let $(x_n,y_n,x'_n)_{n\geq 1} \in S^{-1}(\mathcal{C})$ be a convergent sequence with its limit point $(x_0, y_0,x'_0) \in Z\times X$. Then, it suffices to prove that $ (x_0, y_0,x'_0) \in S^{-1}(\mathcal{C})$. Since $(x_n,y_n,x'_n )\in S^{-1}(\mathcal{C})$, there exists $y'_n\in Z_{x'_n}$, for any $n$, such that
\begin{equation*}
(x'_n, y'_n)\in \mathcal{C} , \qquad \|g(x'_n, y'_n) -  g(x_n,y_n)\|
\leq L_g  d_{X}(x_n, x'_n) .
\end{equation*}
By Lemma \ref{lm:usc}, the sequence $(y'_n)_{n\geq 1}$ has a convergent sub-sequence $(y'_{\varphi(n)})_{n\geq 1}$ with its limit $y'_0\in Z_{x'_0}$. Hence, $\lim_{n\to \infty} (x'_{\varphi(n)}, y'_{\varphi(n)}) = (x'_0,y'_0)$. Since $\mathcal{C}$ is closed, we have $ (x'_0, y'_0)\in \mathcal{C}$. By the triangle inequality,
\begin{equation*}
\begin{split}
\|g(x'_0, y'_0) -  g(x_0,y_0)\| \leq \|g(x'_0, y'_0) - g(x'_{\varphi(n)}, y'_{\varphi(n)})\| + \| g(x'_{\varphi(n)}, y'_{\varphi(n)}) - g(x_{\varphi(n)}, y_{\varphi(n)})\| \\ + \|g(x_{\varphi(n)}, y_{\varphi(n)}) -g(x_0,y_0) \|.
\end{split}
\end{equation*}
By the continuity of $g$, we have
\begin{equation*}
\begin{split}
\|g(x'_0, y'_0) -  g(x_0,y_0)\| &\leq \limsup_{n\to \infty}  \| g(x'_{\varphi(n)}, y'_{\varphi(n)}) - g(x_{\varphi(n)}, y_{\varphi(n)})\|
\\
&\leq \limsup_{n\to \infty}  L_g  d_{X}(x_{\varphi(n)}, x'_{\varphi(n)}) =  L_g  d_{X}(x_0, x'_0).
\end{split}
\end{equation*}
Therefore, $y_0'\in Z_{x_0'}$, $(x_0', y_0')\in \mathcal{C}$ and $ \|g(x'_0, y'_0) -  g(x_0,y_0)\| \leq L_g  d_{X}(x_0, x'_0)$. It follows that
$(x_0,y_0,x_0') \in S^{-1}(\mathcal{C})$, which implies that $S^{-1}(\mathcal{C})$ is closed, thus a Borel set.

Lemma \ref{lm:set_measruable_1} shows that the set-valued function $S$ is Borel measurable, and Theorem \ref{thm:measurable_selection} shows the existence of a Borel measurable selection $s$ of $S$. Since $s(x,y,x')\in S(x,y,x')$, the inequality $\|g(s(x,y,x')) -  g(x,y)\| \leq L_g  d_{X}(x, x')$ is a direct consequence of \eqref{eq:S}.
\end{proof}

\begin{lem}\label{lem:stable}
Let Assumptions \ref{ass1}-\ref{ass2} hold true.
Let $\nu\in \Pi(\bar{\mu}_0, m_1)$ and $s$ be the Borel measurable selection of $S$ obtained Lemma \ref{lm:S1}. Let $\mu_1 = s\# \nu$.
Then $ \mu_1 \in \mathcal{P}_{m_1}(Z)$ and
\begin{equation*}
f\left( \int_{Z} g d \mu_1 \right) - f\left( \int_{Z} g d \bar{\mu}_0 \right) \leq  L_g ( C + L M) \int_{Z\times X} d_{X}(x,x')  d\nu(x,y,x').
\end{equation*}
\end{lem}

\begin{proof}
\textbf{Step 1}. (Properties of $\mu_1$). Since $\nu\in \mathcal{P}(Z\times S)$ and $s\colon Z\times S \to Z$ is a Borel measurable function, we have $\mu = s\# \nu \in \mathcal{P}(Z)$.
Observing that $\pi_1\circ s = \pi_3$, it follows that $
\pi_1 \# \mu = (\pi_1\circ s)\# \nu = \pi_3\# \nu = m_1$. Thus, $\mu_1\in \mathcal{P}_{m_1}(Z)$.
For any $\lambda\in \mathcal{H}$, $g_{\lambda}\colon Z\to \mathbb{R}$ is bounded. Then,
\begin{equation}\label{eq:fund}
\int_{Z} g_{\lambda} d\mu_1 = \int_{Z\times X } g_{\lambda} \circ s \, d\nu .
\end{equation}

\textbf{Step 2}. (Quadratic upper bound). By the Lipschitz continuity of $\nabla f$, we have
\begin{equation}\label{eq:quad1}
f\left( \int_{Z} g d \mu_1 \right) - f\left( \int_{Z} g d \bar{\mu}_0 \right) \leq \int_{Z} g_{\lambda_0} ( d \mu_1 -d \bar{\mu}_0) + \frac{L}{2} \left\| \int_{Z} g d\mu_1 - \int_{Z} g d\bar{\mu}_0 \right\|^2,
\end{equation}
where $\lambda_0 = \nabla f(\int_{Z} g d\bar{\mu}_0)$.

\textbf{Step 3}. (First-order estimate).
Let us study the first-order term in \eqref{eq:quad1}. By \eqref{eq:fund}, we have
\begin{equation*}
\int_{Z} g_{\lambda_0}  d \mu_1 = \int_{Z\times X} g_{\lambda_0} \circ s d\nu = \int_{Z\times X} \langle \lambda_0\, , \, g(s(x,y,x')) \rangle d\nu(x,y,x').
\end{equation*}
From the relation $\pi_{12}\# \nu = \bar{\mu}_0$, we deduce that
\begin{equation*}
\int_{Z} g_{\lambda_0}  d\bar{\mu}_0 = \int_{Z\times X} \langle \lambda_0\, , \, g(x,y) \rangle d\nu(x,y,x').
\end{equation*}
Using the previous two equalities, Lemma \ref{lm:S1}, and the Cauchy–Schwarz inequality, we obtain that
\begin{equation*}
\int_{Z} g_{\lambda_0} ( d \mu_1 -d \bar{\mu}_0) \leq L_g\| \lambda_0\| \int_{Z\times X} d_{X}(x,x')  d\nu(x,y,x').
\end{equation*}

\textbf{Step 4}. (Second-order estimate). Let us study the second-order term in \eqref{eq:quad1}. Developing it and using Lemma \ref{lm:Bochner}, we obtain that
\begin{equation}
\left\| \int_{Z} g d\mu_1 - \int_{Z} g d\bar{\mu}_0 \right\|^2 = \gamma_1 + \gamma_2,
\end{equation}
where
\begin{align*}
\gamma_1 = & \int_{Z} \int_{Z} \langle g(z_1), g(z_2) \rangle \, d (\mu_1-\bar{\mu}_0)(z_1)\, d \mu_1(z_2);\\
\gamma_2 = & \int_{Z} \int_{Z} \langle g(z_1), g(z_2) \rangle \, d (\bar{\mu}_0 - \mu_1)(z_1) \, d \bar{\mu}_0(z_2).
\end{align*}
Fix any $z_2 \in Z$. Following the same argument as in step 3,
we have
\begin{equation*}
\left\| \int_{Z} \langle g(z_1), g(z_2) \rangle \, d (\mu_1-\bar{\mu}_0)(z_1) \right\| \leq L_g\|g(z_2)\| \int_{Z\times X} d_{X}(x,x')  d\nu(x,y,x').
\end{equation*}
It follows that
\begin{equation*}
\gamma_1 + \gamma_2 \leq 2 L_g M \int_{Z\times X} d_{X}(x,x')  d\nu(x,y,x').
\end{equation*}

\textbf{Step 5}. As a consequence of Steps 2-4, we deduce that
\begin{equation*}
f\left( \int_{Z} g d \mu_1 \right) - f\left( \int_{Z} g d \bar{\mu}_0 \right) \leq L_g (\|\lambda_0\| + L M) \int_{Z\times X} d_{X}(x,x')  d\nu(x,y,x').
\end{equation*}
By the definition of the constant $C$, we have $C\geq \|\lambda_0\|$.
The conclusion follows.
\end{proof}

We can now state the bridging algorithm that enables us to obtain an approximate solution of \eqref{pb:primal_1}, given an approximate solution $ \bar{\mu}_0 $ of \eqref{pb:primal_0}.

\medskip

\begin{algorithm}[H]
\caption{Bridging method}
\label{alg:Recover}
\begin{algorithmic}
\STATE{\textbf{Input:} $m_0, m_1 \in \mathcal{P}^1(Z)$, and $\bar{\mu}_0 \in \mathcal{P}_{m_0}(Z)$.}

\STATE{\textbf{Step 1.} Find a solution $\rho$ of the optimal transport problem \eqref{pb:OT1}.}

\STATE{\textbf{Step 2.} Find $\nu \in \mathcal{P}(Z\times X)$ such that $ \pi_{12} \# \nu = \bar{\mu}_0$ and $\pi_{13}\# \nu = \rho$.}

\STATE{\textbf{Step 3.} Set $\mu_1 = s\# \nu\in \mathcal{P}_{m_1}(Z)$, where $s$ is constructed in Lemma \ref{lm:S1}.}

\STATE{\textbf{Output:} $\mu_1$.}
\end{algorithmic}
\end{algorithm}

\smallskip

\begin{rem}\label{rem:empirical}
We have already discussed the case where $m_0$, $m_1$, and $\bar{\mu}_0$ are empirical distributions. We discuss now the slightly more general case where only $m_0$ and $\bar{\mu}_0$ are empirical distributions:
\begin{equation*}
m_0 =  \frac{1}{N}\sum_{i=1}^N \delta_{x_i}, \qquad \bar{\mu}_0 = \frac{1}{N} \sum_{i=1}^N \delta_{(x_i,y_i)}.
\end{equation*}
This situation corresponds to the algorithm presented in Section \ref{sec:algo}.
Since $\rho\in \mathcal{P}_{m_0}(X\times X)$, by Lemma \ref{lm:empirical}, we have $\rho = \frac{1}{N}\sum_{i=1}^N \delta_{x_i}\otimes \rho_{x_i}$, where $\rho_{x_i}$ is defined in Theorem \ref{thm:disintegration}.
Then the probability distribution $\mu$, obtained in general with the Gluing lemma, is given here in an explicit form:
\begin{equation*}
\nu = \frac{1}{N} \sum_{i=1}^N \delta_{(x_i,y_i)}\otimes \rho_{x_i}.
\end{equation*}
\end{rem}

\begin{thm} \label{thm:main}
Let Assumptions \ref{ass1}-\ref{ass2} hold true. Assume that $m_0, m_1 \in \mathcal{P}^1(Z)$ and that $\bar{\mu}_0$ is an $\epsilon_0$-minimizer of problem \eqref{pb:primal_0} for some $\epsilon_0 \geq 0$.
The following holds true.
\begin{enumerate}
\item $| \textnormal{\val}\eqref{pb:primal_0} -\textnormal{\val}\eqref{pb:primal_1} | \leq   L_g ( C + L M) d_1(m_0,m_1)$;
\item  If $\mu_1\in \mathcal{P}_{m_1}(Z)$ is the output of Algorithm \ref{alg:Recover}, then $\mu_1$ is an $\eta$-minimizer of problem \eqref{pb:primal_1}, where
\begin{equation*}
\eta = \epsilon_0 + 2 L_g ( C + L M) d_1(m_0,m_1).
\end{equation*}
\end{enumerate}
\end{thm}

\begin{proof}
We prove (1). Fix any $\epsilon >0$. Let $\mu_0^{\epsilon}$ be an $\epsilon$-minimizer of problem \eqref{pb:primal_0}. By Lemma \ref{lm:OT}, there exists $\nu^{\epsilon} \in \Pi(\mu^{\epsilon}_0,m_1)$ such that
\begin{equation*}
\int_{Z\times X} d_{X}(x,x')  d\nu^{\epsilon}(x,y,x') = d_1(m_0,m_1).
\end{equation*}
We deduce from Lemma \ref{lem:stable} that there exists $\mu_1^{\epsilon}\in \mathcal{P}_{m_1}(Z)$ associated with $\nu^{\epsilon}$ such that
\begin{equation*}
f\left(\int_{z} g\mu_1^{\epsilon} \right) - f\left( \int_{Z} g d \mu_0^{\epsilon} \right) \leq  L_g ( C + L M) d_1(m_0,m_1).
\end{equation*}
Since $\mu_1^{\epsilon} \in  \mathcal{P}_{m_1}(Z)$, $ \textnormal{\val}\eqref{pb:primal_1} \leq  f\left(\int_{z} g\mu_1 \right)$. Combining this with the fact $ \textnormal{\val}\eqref{pb:primal_0} \geq f\left( \int_{Z} g d \mu_0^{\epsilon} \right) - \epsilon$, we obtain that
\begin{equation*}
\textnormal{\val}\eqref{pb:primal_1} -\textnormal{\val}\eqref{pb:primal_0} \leq
f\left(\int_{z} g\mu_1^{\epsilon} \right) - f\left( \int_{Z} g d \mu_0^{\epsilon} \right) + \epsilon.
\end{equation*}
Therefore, $\textnormal{\val}\eqref{pb:primal_1} -\textnormal{\val}\eqref{pb:primal_0} \leq  L_g ( C + L M) d_1(m_0,m_1)$  by the arbitrariness of $\epsilon$. We conclude the first part of the proof by exchanging the positions of $m_0$ and $m_1$.

Let $\mu_1$ be the output of Algorithm \ref{alg:Recover}. To prove (2), we do the following decomposition:
\begin{equation*}
f\left(\int g d \mu_1 \right) - \val\eqref{pb:primal_1} = \gamma_1 + \gamma_2 + \gamma_3,
\end{equation*}
where
\begin{equation*}
\gamma_1  = f\left(\int g d\mu_1 \right) - f\left(\int g d \bar{\mu}_0 \right), \quad
\gamma_2  = f\left(\int g d \bar{\mu}_0 \right) - \val\eqref{pb:primal_0},\quad
\gamma_3  = \val\eqref{pb:primal_0} - \val\eqref{pb:primal_1}.
\end{equation*}
From the proof of Lemma \ref{lm:OT}, we know that $\nu$ (the result of step 2 in Algorithm \ref{alg:Recover}) is a solution of \eqref{pb:OT}. Then, Lemma \ref{lem:stable} shows that $\gamma_1 \leq L_g ( C + L M) d_1(m_0,m_1)$. Since $\bar{\mu}_0$ is an $\epsilon$-minimizer, $\gamma_2\leq \epsilon$. By point (1), $\gamma_3\leq  L_g ( C + L M) d_1(m_0,m_1)$. The conclusion follows.
\end{proof}

\section{Duality analysis} \label{sec:dual}

\subsection{The dual problem}

This section is dedicated to the duality analysis of the primal problem \eqref{pb:primal}. In the sequel of this section, let Assumptions \textbf{A}$^*$ and \ref{ass2} hold true. Consider the equivalent formulation of problem of \eqref{pb:primal},
\begin{equation}\label{pb:primal_eq}\tag{\text{$\tilde{\text{P}}_m$}}
\inf_{\mu \in \mathcal{P}_{m}(Z) ,\,  \beta \in \mathcal{H}} \ f \left( \beta \right), \qquad \text{s.t. } \beta =  \int_{Z} g d \mu.
\end{equation}

The \textit{Lagrangian} $\mathcal{L}\colon \mathcal{H}^2 \times \mathcal{P}_{m}(Z) \to \R$ associated with \eqref{pb:primal_eq} writes,
\begin{equation*}
\mathcal{L} (\lambda,  \beta , \mu ) = f(\beta) + \left\langle \lambda ,  \int_{Z} g d \mu  -  \beta \right\rangle.
\end{equation*}
Then, the dual problem of \eqref{pb:primal_eq} is,
\begin{equation}\label{eq:dual_1}
\sup_{\lambda\in \mathcal{H}}\ \, \underset{\beta \in\mathcal{H}\, ,\, \mu\in\mathcal{P}_{m}(Z)}{ \text{\phantom{p}inf\phantom{p}}} \ \mathcal{L} (\lambda,  \beta , \mu ) = \sup_{\lambda \in \mathcal{H}}  \Big( - f^{*}(\lambda) + \inf_{\mu \in \mathcal{P}_{m}(Z)} \int_{Z} \left\langle \lambda \, , \,
g(z) \right\rangle d \mu(z) \Big),
\end{equation}
where $f^{*}$ is the Fenchel conjugate of $f$. For any $\lambda\in \mathcal{H}$, since $g$ is bounded over $Z$, the second term $ \inf_{\mu \in \mathcal{P}_{m}(Z)} \int_{Z} \left\langle \lambda \, , \,
g(z) \right\rangle d \mu(z)$ is finite. Therefore, it suffices to study \eqref{eq:dual_1} for $\lambda\in \text{dom}(f^{*})$, i.e.,
\begin{equation*}
\sup_{\lambda \in \text{dom}(f^{*})} \Big(  - f^{*}(\lambda) + \inf_{\mu \in \mathcal{P}_{m}(Z)} \int_{Z} g_{\lambda} d \mu \Big).
\end{equation*}
The result of Lemma \ref{lm:inf_int} holds true for all $\lambda\in \text{dom}(f^{*})$ under Assumption \textbf{A}$^*$. Applying it to the previous problem, we obtain the following equivalent dual problem:
\begin{equation}\label{pb:dual}\tag{D$_m$}
- \inf_{\lambda \in \text{dom}(f^{*})} \mathcal{D}_{m}(\lambda) \coloneqq
f^{*}(\lambda) - \int_{X}  u_{\lambda}  d  m.
\end{equation}
\begin{lem}\label{lm:existence}
The function $\mathcal{D}_{m}$ is strongly convex with modulus $1/L$. As a consequence, problem \eqref{pb:dual} has a unique solution, denoted by $\lambda^{*}(m)$. Moreover, there exists a constant $C^{*}$ independent of $m$ such that
\begin{equation*}
\|\lambda^{*}(m)\| \leq C^{*}.
\end{equation*}
\end{lem}
\begin{proof}
Since $\nabla f$ is $L$-Lipschitz continuous, we know that $f^*$ is strongly convex with modulus $1/(2L)$ (i.e.\@ $f^*- 1/L \| \cdot \|^2$ is convex) (see \cite[Thm.\@ 18.15]{bauschke2011convex}). Let us consider $u_{\lambda}(x)$ as a function of $\lambda$ while fixing any $x\in X$. By definition, $\lambda \mapsto u_{\lambda}(x)$ is the infimum of a family of affine functions (with respect to $\lambda)$, thus it is concave with respect to $\lambda$. Consequently, $-\int_{X}u_{\lambda}dm$ is convex with respect to $\lambda$. Therefore, $\mathcal{D}_m$ is $1/L$-strongly convex. Additionally, $\text{dom}(f^*)$ is both convex and closed. These properties guarantee the existence and uniqueness of the minimizer $\lambda^*(m)$.

Since $M$ is an upper bound of $\|g(z)\|$, it follows that for all $\lambda\in\mathcal{H}$:
\begin{equation*}
-M\|\lambda\|\leq \inf_{y\in Z_{x}}-\|\lambda\|\|g(x,y)\|\leq u_{\lambda}(x)\leq \sup_{y\in Z_{x}}\|\lambda\|\|g(x,y)\|\leq M\|\lambda\|.
\end{equation*}
Let $\lambda_0\in\text{dom}(f^*)$. As $\mathcal{D}_m(\lambda^*(m))\leq \mathcal{D}_m(\lambda_0)$, we can derive the following inequalities:
\begin{equation*}
f^*(\lambda_0)+M\|\lambda_0\|\geq \mathcal{D}_m(\lambda_0)\geq \mathcal{D}_m(\lambda^*(m))\geq f^*(\lambda^*(m))-M\|\lambda^*(m)\|.
\end{equation*}
The strong convexity of $f^*$ yields that
\begin{equation*}
\frac{1}{2L}\|\lambda^*(m)-\lambda_0\|^2+\langle p_0,\lambda^*(m)-\lambda_0\rangle\leq f^*(\lambda^*(m))-f^*(\lambda_0),
\end{equation*}
where $p_0\in\partial f^*(\lambda_0)$.
Combining the two above inequalities, we obtain:
\begin{equation*}
\frac{1}{2L}\|\lambda^*(m)-\lambda_0\|^2+\langle p_0,\lambda^*(m)-\lambda_0\rangle\leq M(\|\lambda^*(m)\|+\|\lambda_0\|).
\end{equation*}
where $p_0\in\partial f^*(\lambda_0)$. The announced result follows, with $C^*=3\|\lambda_0\|+2L(M+\|p_0\|)$.
\end{proof}

\subsection{Strong duality}
Let us now prove the strong duality principle between \eqref{pb:primal} and \eqref{pb:dual}, i.e., $\textnormal{\val}\eqref{pb:primal} = \textnormal{\val}\eqref{pb:dual}$ . We will apply the Fenchel-Rockafellar theorem \cite{rockafellar1997convex} to prove this relation.
\begin{prop}\label{prop:strong_duality}
Assume that the set $G_m \coloneqq \{ \int_{\mu}g d\mu \mid  \mu\in \mathcal{P}_m(Z)\} \subseteq \mathcal{H}$ is closed. Then,
\begin{enumerate}
\item $\textnormal{\val}\eqref{pb:primal} = \textnormal{\val}\eqref{pb:dual}$;
\item the primal problem \eqref{pb:primal} has a solution;
\item let $\lambda^{*}(m)$ be the solution of \eqref{pb:dual} and let $\mu$ be a solution of \eqref{pb:primal}, then
\begin{equation*}
\lambda^{*}(m) = \nabla f\left(\int_{Z} g d\mu\right).
\end{equation*}
\end{enumerate}
\end{prop}
\begin{proof}
Let us consider the following optimization problem with variable in $\mathcal{H}$:
\begin{equation}\label{pb:pri_opt}
\inf_{z \in \mathcal{H} } f(z) + \chi_{G_m}(z).
\end{equation}
It is obvious that $\val\eqref{pb:primal} = \val\eqref{pb:primal_eq} = \val\eqref{pb:pri_opt}$. The dual problem of \eqref{pb:pri_opt} writes
\begin{equation}\label{pb:dual_opt}
\sup_{\lambda \in \mathcal{H} } -f^{*}(\lambda) - \chi_{G_m}^{*}(-\lambda).
\end{equation}
By the definition of the Fenchel conjugate and the definition of $G_m$, we have
\begin{equation*}
- \chi_{G_m}^{*}(-\lambda) = \inf_{z\in G_m} \langle \lambda, z\rangle = \inf_{\mu\in \mathcal{P}_m(Z)} \left\langle \lambda, \int_{Z} g d\mu \right\rangle.
\end{equation*}
Therefore, $\val\eqref{pb:dual} = \val\eqref{eq:dual_1} = \val\eqref{pb:dual_opt}$. Let us apply the Fenchel-Rockafellar theorem to \eqref{pb:pri_opt}. The function $f$ is convex and continuous. The function $\chi_{G_m}$ is convex and lower semi-continuous from the fact that $G_m$ is convex and closed. It is obvious that $G_m$ is non-empty. Therefore, $0\in \text{int} (\mathcal{H}- G_m) =\text{int} (\text{dom} f- \text{dom} \chi_{G_m}) $. By the Fenchel-Rockafellar theorem \cite{rockafellar1997convex}, $\val\eqref{pb:dual_opt} = \val\eqref{pb:pri_opt}$, thus, $\textnormal{\val}\eqref{pb:primal} = \textnormal{\val}\eqref{pb:dual}$.

Since $G_m$ is non-empty, bounded, convex, and closed, and since $f$ is continuous and convex, we deduce from \cite[Cor.\@ 3.23]{brezis2011functional} that problem \eqref{pb:pri_opt} has a solution. Therefore, \eqref{pb:primal} has a solution, denoted by $\mu$. Since $\lambda^{*}(m)$ is the solution of \eqref{pb:dual}, by the strong duality,
\begin{equation*}
-f^{*}(\lambda^{*}(m)) + \inf_{z\in G_m} \langle \lambda^{*}(m), z\rangle  =  f\left(\int_{Z} g d\mu\right).
\end{equation*}
On the other hand, by the definition of Fenchel's conjugate,
\begin{equation*}
f\left(\int_{Z} g d\mu\right) + f^{*}(\lambda^{*}(m)) \leq \left\langle \lambda^{*}(m),  \int_{Z} g d\mu\right\rangle .
\end{equation*}
Combining the previous two inequalities and the fact that $ \int_{Z} g d\mu \in G_m $, we deduce that
\begin{equation*}
f\left(\int_{Z} g d\mu\right) + f^{*}(\lambda^{*}(m)) = \left\langle \lambda^{*}(m),  \int_{Z} g d\mu\right\rangle.
\end{equation*}
We obtain that $\lambda^{*} (m)= \nabla f(\int_Z g d\mu) $ from Fenchel's relation.
\end{proof}

\subsection{Stability of the dual solution}

\begin{lem}\label{lm:dual_lip}
For any $\lambda_1,\lambda_2 \in \mathcal{H}$ and $x_1,x_2\in X$, it holds that
\begin{equation*}
|u_{\lambda_1}(x_1) - u_{\lambda_2}(x_2)|\leq L_g  \|\lambda_1\| d_X(x_1,x_2) + M \| \lambda_1 - \lambda_2\| .
\end{equation*}
\end{lem}
\begin{proof}
By the triangle ineqaulity,
\begin{equation*}
|u_{\lambda_1}(x_1) - u_{\lambda_2}(x_2)| \leq |u_{\lambda_1}(x_1) - u_{\lambda_1}(x_2)| + |u_{\lambda_1}(x_2) - u_{\lambda_2}(x_2)|.
\end{equation*}
By the definition of $u_{\lambda}$, we have
\begin{equation*}
u_{\lambda_1} (x_1) - u_{\lambda_1}(x_2) = \inf_{y_1\in Z_{x_1}} \langle \lambda_1 \, , \, g(x_1, y_1)\rangle - \inf_{y_2\in Z_{x_2}} \langle \lambda_1 \, , \, g(x_2, y_2)\rangle.
\end{equation*}
Let $\tilde{y}^{\epsilon}_2$ be an $\epsilon$-minimizer of $\inf_{y_2\in Z_{x_2}} \langle \lambda_1 \, , \, g(x_2, y_2)\rangle$, with $\epsilon >0$. By the Lipschitz continuity of $\mathcal{Z}$, there exists $\tilde{y}_1^{\epsilon} \in Z_{x_1}$ such that
\begin{equation*}
\|g(x_1, \tilde{y}_1^{\epsilon}) - g(x_2, \tilde{y}_2^{\epsilon})\| \leq L_gd_X(x_1,x_2).
\end{equation*}
By the Cauchy-Schwarz inequality, we have
\begin{equation*}
u_{\lambda_1} (x_1) - u_{\lambda_1}(x_2) \leq  \langle \lambda_1 \, , \, g(x_1, \tilde{y}_1^{\epsilon}) - g(x_2, \tilde{y}_2^{\epsilon})\rangle + \epsilon \leq \|\lambda_1\| L_gd_X(x_1,x_2) +\epsilon.
\end{equation*}
By the arbitrariness of $\epsilon$, we have $ |u_{\lambda_1} (x_1) - u_{\lambda_1}(x_2)| \leq \|\lambda_1\| L_gd_X(x_1,x_2)$.

On the other hand,
\begin{equation*}
u_{\lambda_1} (x_2) - u_{\lambda_2}(x_2) = \inf_{y \in Z_{x_2}} \langle \lambda_1 \, , \, g(x_2, y)\rangle - \inf_{y\in Z_{x_2}} \langle \lambda_2 \, , \, g(x_2, y )\rangle \leq \sup_{y\in Z_{x_2}} \langle \lambda_1 - \lambda_2 \, , \, g(x_2, y )\rangle.
\end{equation*}
By the Cauchy–Schwarz inequality and the definition of $M$, we have that
\begin{equation*}
u_{\lambda_1} (x_2) - u_{\lambda_2}(x_2) \leq M \|\lambda_1 - \lambda_2\|.
\end{equation*}
The conclusion follows.
\end{proof}

\begin{lem}[Stability of the dual problem]\label{lm:dual_sta}
For any $m_0$, $m_1 \in \mathcal{P}(\Omega)$, we have
\begin{align}
\left| \mathcal{D}_{m_0}(\lambda^{*}(m_0)) - \mathcal{D}_{m_1}(\lambda^{*}(m_1)) \right| & \leq C^{*}L_g \, d_1(m_0,m_1), \label{eq:stable_1}
\\
\|\lambda^{*}(m_0) - \lambda^{*}(m_1)\|^2 & \leq 2 C^{*}L_g L \, d_1(m_0,m_1), \label{eq:stable_2}
\end{align}
where $C^{*}$ is the a priori bound of $\|\lambda^{*}(\cdot)\|$ obtained in Lemma \ref{lm:existence}
\end{lem}

\begin{proof}
According to Lemma \ref{lm:existence}, we know that $\|\lambda^{*}(m_0)\|$ and $\|\lambda^{*}(m_1)\|$ are smaller than $C^{*}$. Then, by Lemma \ref{lm:dual_lip}, $u_{\lambda^{*}(m_0)}(x)$ and $u_{\lambda^{*}(m_1)}(x)$ are $(C^{*}L_g)$-Lipschitz continuous with respect to $x$. Hence,
\begin{equation}\label{eq:stability_1}
\begin{split}
\mathcal{D}_{m_0}(\lambda^{*}(m_0)) = {} & f^{*}(\lambda^{*}(m_0)) - \int_{X} u_{\lambda^{*}(m_0)} (x) dm_0 (x) \\
= {} &  f^{*}(\lambda^{*}(m_0)) - \int_{X} u_{\lambda^{*}(m_0)} (x) dm_1(x) +  \int_{X} u_{\lambda^{*}(m_0)} (x) d(m_1-m_0)(x) \\
\geq {} & \mathcal{D}_{m_1}(\lambda^*(m_0)) -  C^{*}L_g d_1(m_0,m_1),
\end{split}
\end{equation}
where the third line is by the definition of the Kantorovich–Rubinstein distance.
Since $\lambda^*(m_1)$ minimizes $\mathcal{D}_{m_1}$ and since $\mathcal{D}_{m_1}$ is $1/L$-strongly convex, we have
\begin{equation} \label{eq:stability_1bis}
\mathcal{D}_{m_1}(\lambda^*(m_0))
\geq
\mathcal{D}_{m_1}(\lambda^{*}(m_1)) + \frac{1}{2L} \|\lambda^{*}(m_0) - \lambda^{*}(m_1)\|^2.
\end{equation}
Combining \eqref{eq:stability_1} and \eqref{eq:stability_1bis}, we obtain that
\begin{equation*}
\mathcal{D}_{m_0}(\lambda^{*}(m_0)) \geq \mathcal{D}_{m_1}(\lambda^{*}(m_1)) + \frac{1}{2L} \|\lambda^{*}(m_0) - \lambda^{*}(m_1)\|^2 -  C^{*}L_g d_1(m_0,m_1).
\end{equation*}
In particular, we have $ \mathcal{D}_{m_1}(\lambda^{*}(m_1)) - \mathcal{D}_{m_0}(\lambda^{*}(m_0)) \leq  C^{*}L_g d_1(m_0,m_1)$.  Exchanging the positions of $m_0$ and $m_1$ in \eqref{eq:stability_1}, we obtain
\begin{equation}\label{eq:stability_2}
\mathcal{D}_{m_1}(\lambda^{*}(m_1)) \geq \mathcal{D}_{m_0}(\lambda^{*}(m_0))+ \frac{1}{2L} \|\lambda^{*}(m_0) - \lambda^{*}(m_1)\|^2  - C^{*}L_g d_1(m_0,m_1).
\end{equation}
Inequality \eqref{eq:stable_1} follows immediately and \eqref{eq:stable_2} is deduced by summing \eqref{eq:stability_1}-\eqref{eq:stability_2}.
\end{proof}

\subsection{Directional derivative of the value function}

The value function of problem \eqref{pb:primal} is defined by
\begin{equation*}
V\colon \mathcal{P}^1(X) \to \R, \ m\mapsto \val\eqref{pb:primal}.
\end{equation*}
Our goal is to characterize the directional derivative of $V$. Define the following function:
\begin{align*}
v \colon \mathcal{P}^1(X)\times X \to \R, \ (m,x) \mapsto u_{\lambda^{*}(m)} (x).
\end{align*}

\begin{prop}
Assume that $G_m$ is closed for any $m \in \mathcal{P}^1(X)$. Then for any $m_0, m_1\in \mathcal{P}^1(X)$, we have
\begin{equation*}
\lim_{t\to 0^{+}} \frac{V(m_0 + t (m_1-m_0)) - V(m_0)}{t} = \int_{X} v(m_0,x) d (m_1-m_0)(x).
\end{equation*}
As a consequence, $v$ is the directional derivative of $V$, i.e.,
\begin{equation*}
V(m_1) - V(m_0) = \int_{t=0}^1 \int_{X} v(m_0+t(m_1-m_0), x ) d (m_1-m_0)(x) dt.
\end{equation*}
\end{prop}

\begin{proof}
For any $t\in[0,1]$, let $m_t = m_0 + t (m_1-m_0)$.
By the strong duality, we have
\begin{equation*}
V(m_t) - V(m_0) = \mathcal{D}_{m_0}(\lambda^{*}(m_0)) - \mathcal{D}_{m_t}(\lambda^{*}(m_t)).
\end{equation*}
From \eqref{eq:stability_1}, we deduce that
\begin{equation}\label{eq:dual1}
\mathcal{D}_{m_0}(\lambda^{*}(m_0)) - \mathcal{D}_{m_t}(\lambda^{*}(m_t)) \geq \int_{X} v(m_0,x) d (m_t-m_0)(x) = t  \int_{X} v(m_0,x) d (m_1-m_0)(x).
\end{equation}
On the other hand, let $\mu_0$ and $\mu_1$ be solutions of \eqref{pb:primal} with $m=m_0$ and $m_1$ respectively. Let $\mu_t = \mu_0 + t(\mu_1-\mu_0)$.
It is obvious that $\pi_1\# \mu_t = m_t$. Therefore,
\begin{equation*}
V(m_t)-V(m_0) \leq f\left(\int_{Z} g d \mu_t\right) -f\left( \int_{Z} g d\mu_0\right).
\end{equation*}
By Proposition \ref{prop:strong_duality}, $\lambda^{*}(m_0) = \nabla f (\int_{Z} g d\mu_0)$.
Since $\nabla f$ is $L$-Lipschitz, it follows that
\begin{equation*}
f\left(\int_{Z} g d \mu_t\right) -f\left( \int_{Z} g d\mu_0\right) \leq  t  \int_{X} v(m_0,x) d (m_1-m_0)(x) + \frac{Lt^2}{2}\left\| \int_{Z} g d(\mu_1 - \mu_0) \right\|^2.
\end{equation*}
Recall the definition of $D$. Combining the two inequalities above, we have
\begin{equation}\label{eq:dual_2}
V(m_t)-V(m_0) \leq  t  \int_{X} v(m_0,x) d (m_1-m_0)(x) + \frac{LDt^2}{2}.
\end{equation}
We have $V(m_0)= -\mathcal{D}_{m_0}(\lambda^*(m_0))$ and $V(m_t)= -\mathcal{D}_{m_t}(\lambda^*(m_t))$.
Using \eqref{eq:dual1}-\eqref{eq:dual_2} and letting $t$ go to $0^{+}$, we obtain that
\begin{equation*}
\lim_{t\to 0^{+}} \frac{V(m_t) - V(m_0)}{t} = \int_{X} v(m_0,x) d (m_1-m_0)(x).
\end{equation*}

From Lemmas \ref{lm:dual_lip}-\ref{lm:dual_sta}, we deduce that the function $v(m,x)$ is continuous in $\mathcal{P}^1(X)\times X$ with respect to the distance $(d_1, d_X)$. Let us define two functions from $[0,1]$ to $\R$,
\begin{align*}
& \bar{V}\colon [0,1] \to \R, \, t \mapsto V(m_t);\\
& \bar{v} \colon [0,1] \to \R, \, t \mapsto \int_{X} v(m_t,x) d (m_1-m_0)(x).
\end{align*}
For any $0\leq t \leq T < 1$,
observe that $m_T = m_t + \frac{T-t}{1-t} (m_1-m_t)$ and $m_1 - m_t = (1-t) (m_1-m_0)$.
By using the same arguments as in \eqref{eq:dual1}-\eqref{eq:dual_2}, we have
\begin{equation*}
(T-t) \bar{v}(t) \leq \bar{V}(T) -\bar{V}(t) \leq (T-t) \bar{v}(t) + \frac{LD(T-t)^2}{2(1-t)^2}.
\end{equation*}
We deduce that $\bar{v}(t)$ is the right derivative of $\bar{V}$ at $t$ for any $t\in [0,1)$. By exchanging positions of $m_0$ and $m_1$, we can prove that $\bar{v}(t)$ is the left derivative of $\bar{V}$ at $t$ for any $t\in (0,1]$. Therefore, $\bar{V}$ is differentiable at each point on $[0,1]$ and $\bar{v}$ is its derivative. Since $\bar{v}$ is continuous, by the fundamental theorem of calculus \cite[Thm.\@ 7.21]{rudin}, we have that $\bar{V}(1) - \bar{V}(0) = \int_{t=0}^{1} \bar{v}(t) dt$.
\end{proof}

\section{Numerical approach} \label{sec:algo}

We present in this section our numerical method for solving \eqref{pb:primal}. The first step of resolution consists in discretizing $m$. We replace it by an empirical distribution $m_N$ and focus next on the resolution of \eqref{pb:primal_dis}.
By Theorem \ref{thm:main}(1), we have
\begin{equation} \label{eq:estimate_disc}
| \textbf{val}\eqref{pb:primal} - \textbf{val}\eqref{pb:primal_dis} | \leq L_g(C+LM) d_1(m,m_N).
\end{equation}
We give theoretical bounds for the minimal value of $d_1(m,m_N)$ in Subsection \ref{subsec:discretization}. Then we discuss the resolution of \eqref{pb:primal_dis} with the Frank-Wolfe algorithm in Subsection \ref{subsec:fw_algo}. Finally in Subsection \ref{subsec:sfw_algo} we propose to use a variant of the Frank-Wolfe algorithm, called Stochastic Frank-Wolfe (SFW) algorithm, introduced in \cite{bonnans2022large}. This method generates a solution to \eqref{pb:primal_dis} which is an empirical distribution.

\subsection{Discretization} \label{subsec:discretization}

In view of \eqref{eq:estimate_disc}, one should look for an empirical distribution
$ m_N = \frac{1}{N} \sum_{i=1}^N \delta_{x_i}$ that is as close as possible to $m$ for the $d_1$-distance.
This problem is commonly known as the optimal quantization problem, and for detailed information on this topic, we refer to \cite{gersho2012vector}.
Here, we present a slightly modified version of an optimal quantization result obtained in \cite[Prop.\@ 12]{merigot2016minimal}. For any subset $A$ of $X$, we denote by $r_N(A)$ the minimum radius $r$ required to cover $A$ with $N$ closed balls of radius $r$. It is defined by
\begin{equation*}
r_N (A)\coloneqq \inf_{x\in A^N} \min \left\{r\geq 0 \, \Big| \, A \subseteq \bigcup_{i=1}^N B_X(x_i,r) \right\}.
\end{equation*}
The upper box-counting dimension (or the upper Minkowski dimension) of $A$ \cite[p.\@ 41]{falconer2004fractal} is defined as follows:
\begin{equation*}
\bar{D}(A) \coloneqq \inf\left\{ \bar{D}>0 \; \big|\; \exists \bar{C}>0 \text{ such that } r_N(A)\leq \bar{C}N^{-1/\bar{D}} ,\; \forall N\in \mathbb{N}_{+} \right\}.
\end{equation*}

\begin{lem}\label{lm:quantization}
Let $m\in \mathcal{P}^1(X)$, and let $A\subseteq X$ be the support of $m$.
There exists
a sequence $(m_N)_{N\geq 1}$ of empirical distributions on $X$ such that the following holds:
\begin{enumerate}
\item If $\bar{D}(A)>1$, then there exists a constant $\tilde{C}_1$ such that for any $N\geq 1$,
\begin{equation*}
d_1(m,m_N) \leq \tilde{C}_1 N^{-\frac{1}{\bar{D}(A)}}.
\end{equation*}
\item If $\bar{D}(A)=1$, then there exists a constant $\tilde{C}_2$ such that for any $N\geq 1$,
\begin{equation*}
d_1(m,m_N) \leq \tilde{C}_2 N^{-1}\log N.
\end{equation*}
\item If $\bar{D}(A)<1$, then there exists a constant $\tilde{C}_3$ such that for any $N\geq 1$,
\begin{equation*}
d_1(m,m_N) \leq \tilde{C}_3 N^{-1}.
\end{equation*}
\end{enumerate}
\end{lem}

\begin{proof}
This follows from the proof presented in \cite[Prop.\@ 12]{merigot2016minimal}, with the only difference being that in the final inequality, we employ the triangle inequality for the $d_1$-distance instead of the Minkowski inequality for the Wasserstein-2 distance.
\end{proof}

\begin{rem}
If $A$ is a subset of a smooth $d$-dimensional submanifold of a Euclidean space, then $\bar{D}(A) \leq  d$. This estimate is deduced from \cite[p.\@ 48 (i)-(ii)]{falconer2004fractal}.
\end{rem}

\subsection{Frank-Wolfe algorithm}
\label{subsec:fw_algo}

For general convex optimization problems, the Frank-Wolfe algorithm relies on the resolution of a sequence of linearized problems, obtained by replacing the cost function of the problem by a first-order Taylor approximation of it. In the context of problem \eqref{pb:primal_dis}, the linearized problem is of the general form:
\begin{equation}\label{pb:linear}
\inf_{\mu\in \mathcal{P}_{m_N}(Z)} \left\langle  \lambda , \int_{Z} g d\mu\right\rangle,
\end{equation}
for some $\lambda \in  \nabla f(\mathcal{H})$.

A key observation from Lemma \ref{lm:inf_int} is that a solution of the linearized problem, denoted by $\mu_{\lambda}$, can be obtained as in the proof of Lemma \ref{lm:inf_int}, in the simple case where $m$ is a finitely-supported probability measure: for all $i=1,\ldots,N$, find $y_i \in \br_{\lambda}(x_i)$ and set ${\mu}_\lambda= \frac{1}{N} \sum_{i=1}^N \delta_{(x_i,y_i)}$.
Therefore, one can consider applying the Frank-Wolfe algorithm to solve \eqref{pb:primal_dis}, in which the main task is to solve \eqref{pb:linear}.

\smallskip

\begin{algorithm}[H]
\caption{Frank-Wolfe Algorithm }
\label{alg:FW}
\begin{algorithmic}
\STATE{Initialization: $\mu^0 \in \mathcal{P}_{m_N}(Z).$ Set $K\geq 1$.}
\FOR{$k= 0,1,2,\ldots ,K-1$}

\STATE{Compute $\lambda^k = \nabla f \left( \int_{Z} g d\mu^k \right)$.}

\STATE{Solve \eqref{pb:linear} for $\lambda=\lambda^k$, the solution is denoted by $\mu_{\lambda^k}$.}

\STATE{
{Choose $\omega_k \in [0,1]$.} }

\STATE{Set $\mu^{k+1} = (1-\omega_k)\mu^k + \omega_k \mu_{\lambda^{k}}$.}

\ENDFOR
\end{algorithmic}
\end{algorithm}

\begin{rem}
If we take $\omega_k = 1/(k+1)$ for all $k$, then it is easy to see that $\mu^{K} = \frac{1}{K}\sum_{k=0}^{K-1} \mu_{\lambda^k}$. We recover the fictitious play of \cite{cardaliaguet2017learning}, applied to the Lagrangian discretization of first-order MFGs.
\end{rem}

\begin{lem}
Let Assumption \ref{ass1} hold true. In Algorithm \ref{alg:FW}, we set $\omega_k = 2/(k+2)$ for all $k$. Then for any $K\geq 1$,
\begin{equation*}
f\left(\int_{Z} g d\mu^K \right) - \textnormal{\textbf{val}}\eqref{pb:primal_dis} \leq \frac{2LD}{K}.
\end{equation*}
\end{lem}

\begin{proof}
This is a consequence of \cite[Prop.\@ 3.4]{bonnans2022large}.
\end{proof}

\subsection{Stochastic Frank-Wolfe algorithm}
\label{subsec:sfw_algo}

In Algorithm \ref{alg:FW}, at each iteration, we generate the output by taking a convex combination of the previous iteration's result and the solution of \eqref{pb:linear}. This process requires us to add $N$ points from $\br_{\lambda^k}(x_i)$, for $i=1,\ldots,N$, to stock the support of solution at each iteration. As a consequence, this approach can lead to a  memory overflow issue,  as $K$ going to infinity. The large support of $\mu^K$ will also raise the difficulty of Step 2 in Algorithm \ref{alg:Recover}, in which we will take $\bar{\mu}_0 = \mu^K$.
To address this issue, we will use the stochastic Frank-Wolfe algorithm \cite{bonnans2022large} to \eqref{pb:primal_dis}. This approach will enable us to obtain an approximate empirical solution \eqref{pb:primal_dis}, and can effectively handle the large support of $\mu^K$.

\begin{lem}\label{lm:empirical}
Let $\mu \in \mathcal{P}(Z)$. Then
$\mu$ lies in $\mathcal{P}_{m_N}(Z)$ if and only if there exists $\mu_i \in \mathcal{P}(Z_{x_i})$ for any $i=1,\ldots, N$ such that $\mu = \frac{1}{N}\sum_{i=1}^N \delta_{x_i}\otimes \mu_i$.
\end{lem}

\begin{proof}
If $\mu_i\in \mathcal{P}( Z_{x_i})$, then $\pi_1\#( \delta_{x_i}\otimes \mu_i )= \delta_{x_i}$. Since the push-forward operator $\#$ is linear, we have that $\pi_1\#(\frac{1}{N}\sum_{i=1}^N \delta_{x_i}\otimes  \mu_i) = \frac{1}{N}\sum_{i=1}^N \delta_{x_i} = m_N$.

Conversely, let us assume that $\mu \in \mathcal{P}_{m_N}(Z)$. By Theorem \ref{thm:disintegration} and its remark, we can conclude that there exists $\mu_{x_i}\in \mathcal{P}(Z_{x_i})$ for $i=1,\ldots, N$ such that for any bounded and continuous function $h\colon Z\to \R$, we have
\begin{equation*}
\int_{Z}h d\mu = \frac{1}{N} \sum_{i=1}^N \int_{Z_{x_i}} h(x_i, y_i) d\mu_{x_i}(y_i).
\end{equation*}
Applying Fubini's theorem to the equality above, we have
\begin{equation*}
\int_{Z}h \, d\mu = \frac{1}{N}
\sum_{i=1}^N \int_{Z} h \, d \left( \delta_{x_i}\otimes \mu_{x_i}\right) = \int_{Z} h \, d \left( \frac{1}{N} \sum_{i=1}^N \delta_{x_i}\otimes \mu_{x_i}\right).
\end{equation*}
This implies that $\mu = \frac{1}{N} \sum_{i=1}^N \delta_{x_i}\otimes  \mu_{x_i}$.
\end{proof}

According to Lemma \ref{lm:empirical} and Fubini's theorem, the discretized problem \eqref{pb:primal_dis} is equivalent to
\begin{equation}\label{pb:dirac_rex}
\inf_{\mu_i \in \mathcal{P}(Z_{x_i})}   f \left( \frac{1}{N} \sum_{i=1}^N \int_{Z_{x_i}} g(x_i,y_i) d\mu_i(y_i) \right).
\end{equation}
Problem \eqref{pb:dirac_rex} is the randomized relaxation of an $N$-agent optimization problem as investigated in \cite{bonnans2022large},
\begin{equation}\label{pb:dirac}
\inf_{y\in \prod_{i=1}^N Z_{x_i}}   f \left( \frac{1}{N} \sum_{i=1}^N g(x_i,y_i) \right).
\end{equation}
Problem \eqref{pb:dirac} is equivalent to a version of problem \eqref{pb:dirac_rex} in which the probability measures $\mu_i$ are restricted to be Dirac measures. In particular, we can associate with each feasible element $y=(y_i)_{i=1,\ldots,N} \in \prod_{i=1}^N Z_{x_i}$ (for problem \eqref{pb:dirac}) the tuple $(\delta_{x_i})_{i=1,\ldots,N}$, which is feasible for \eqref{pb:dirac_rex}, and the probability distribution $\frac{1}{N} \sum_{i=1}^N \delta_{(x_i,y_i)}$, which is feasible for \eqref{pb:primal_dis}.

We apply the following Stochastic Frank-Wolfe algorithm, investigated in \cite{bonnans2022large}, to solve problems \eqref{pb:dirac_rex} and \eqref{pb:dirac}.  Let Bern$(\omega)$ be the Bernoulli distribution with a parameter $\omega\in [0,1]$.

\medskip

\begin{algorithm}[H]
\caption{Stochastic Frank-Wolfe Algorithm }
\label{alg:SFW}
\begin{algorithmic}
\STATE{Initialization: $ y^0 \in \prod_{i=1}^N Z_{x_i}.$ Set $K\geq 1$.}
\FOR{$k= 0,1,2,\ldots ,K-1$}

\STATE{Compute $\lambda^k = \nabla  f(\frac{1}{N} \sum_{i=1}^N   {g} (x_i, y^k_i))$.}
\FOR{$i=1,2,\ldots,N$}

\STATE{
Find $\bar{y}_i^k \in \br_{\lambda^k}(x_i)$.
}

\ENDFOR

\STATE{Choose $n_k \in \mathbb{N}^*$.
{Set $\omega_k = 2/(k+2)$.} }
\FOR{$j=1,2,\ldots,n_k$}

\FOR{$i=1,2,\ldots,N$}
\STATE{Simulate $P^{k,j}_i \sim \text{Bern}(\omega_k)$, independently of all previously defined random variables.}
\STATE{Set $\hat{y}_i^{k,j} = (1-P^{k,j}_i) y_i^k + P^{k,j}_i \bar{y}_i^k$.}
\ENDFOR
\STATE{Define $\hat{y}^{k,j}= (\hat{y}^{k,j}_i)_{i=1,\ldots,N}$.}

\ENDFOR
\STATE{Find $y^{k+1} \in \argmin \big\{ f(\frac{1}{N} \sum_{i=1}^N   {g} (x_i, y_i)) \, \big| \, y \in \{ \hat{y}^{k,j},\, j=1,2,\ldots,n_k \} \big\}$.}

\ENDFOR

\end{algorithmic}
\end{algorithm}

\medskip

The interest of Algorithm \ref{alg:SFW} is that it provides an approximate solution to \eqref{pb:dirac}, and the associated empirical distribution serves as a reliable approximate solution of the problem \eqref{pb:dirac_rex}, as demonstrated in the following lemma. Additionally, this empirical distribution has a fixed support size $N$, which does not increase with the iteration number, making the algorithm memory-efficient.

\begin{lem}\label{lm:convergence}
in Algorithm \ref{alg:SFW}, whatever the numbers $(n_k)_{k \in \mathbb
N}$, we have for any $K=1,2,\ldots, 2N$ that
\begin{equation*}
\mathbb{E}\left[ f\left(\frac{1}{N} \sum_{i=1}^N   {g} (x_i, y_i^K)\right)  \right] - \textnormal{\textbf{val}}\eqref{pb:dirac_rex} \leq \frac{4LD}{K}.
\end{equation*}
\end{lem}

\begin{proof}
This is from \cite[Thm.\@ 3.7]{bonnans2022large}.
\end{proof}

\begin{rem}
Lemma \ref{lm:convergence} provides a convergence result for Algorithm \ref{alg:SFW} in terms of expectation.
An estimate of the following quantity can be found in \cite[Thm.\@ 3.7]{bonnans2022large}:
\begin{equation*}
\mathbb{P}\left[ f\left(\frac{1}{N} \sum_{i=1}^N   {g} (x_i, y_i^K)\right)  \geq \textnormal{\textbf{val}}\eqref{pb:dirac_rex} + \epsilon + \frac{4LD}{K}\right],
\end{equation*}
for a given $\epsilon >0$. In particular, this probability can be made arbitrarily small, provided that the numbers $n_k$ are large enough.
\end{rem}

In order to obtain an approximate solution of \eqref{pb:primal}, we combine Algorithm \ref{alg:SFW} with Algorithm \ref{alg:Recover}.
Let us consider the outcome $y^K$ of Algorithm \ref{alg:SFW} after $K$ iterations, for $1 \leq K \leq 2N$ and for arbitrary numbers $n_k \geq 1$ of simulations.
Let $\mu^K_N = \frac{1}{N}\sum_{i=1}^N \delta_{(x_i,y_i^K)}$.
Moving on to Algorithm \ref{alg:Recover}, we utilize the following inputs: $m_0 = m_N$, $m_1 = m$, and $\bar{\mu}_0 = \mu_N^K$. The output of Algorithm \ref{alg:Recover} is denoted as $\tilde{\mu}^K$, which is an element of the set $\mathcal{P}_m(Z)$.
We have the following convergence result for the combination of Algorithm \ref{alg:Recover} and \ref{alg:SFW}.

\begin{thm}\label{thm:main2}
Let Assumptions \ref{ass1}-\ref{ass2} hold true, and let $m\in \mathcal{P}^1(Z)$. Then,
\begin{equation*}
\mathbb{E}\left[f\left(\int_{Z}g d \tilde{\mu}^K\right)\right] - \textnormal{\textbf{val}}\eqref{pb:primal} \leq \frac{4LD}{K} + 2 L_g ( C + L M) d_1(m_N,m).
\end{equation*}
\end{thm}

\begin{proof}
Since
$ f\left( \int_{Z} g d\mu_N^K \right) = f\left( \frac{1}{N}\sum_{i=1}^N g(x_i,y^K_i)\right)$,
by Lemma \ref{lem:stable}, we have
\begin{equation*}
f\left( \int_{Z} g \tilde{\mu}^K \right) - f\left( \frac{1}{N}\sum_{i=1}^N g(x_i,y^K_i)\right) \leq L_g ( C + L M) d_1(m_N,m), \qquad \text{almost surely}.
\end{equation*}
Taking expectation on both sides of the previous inequality, and applying Lemma \ref{lm:convergence} and the relation $\textbf{val}\eqref{pb:dirac_rex}= \textbf{val}\eqref{pb:primal_dis}$, we have
\begin{equation*}
\mathbb{E}\left[f\left(\int_{Z}g d \tilde{\mu}^K\right)\right] - \textbf{val}\eqref{pb:primal_dis}\leq  \frac{2LD}{K} +  L_g ( C + L M) d_1(m_N,m).
\end{equation*}
Combining with Theorem \ref{thm:main}(1), the proof is complete.
\end{proof}

\begin{rem}
The realization of Algorithm \ref{alg:Recover} can be simplified in Theorem \ref{thm:main2} thanks to the empirical structure of $m_0$ and $\bar{\mu}_0$, as noted in Remark \ref{rem:empirical}.
\end{rem}

\section{Examples and numerical results}
\label{sec:examples}

\subsection{The traffic assignment problem}

The traffic assignment problem is a non-atomic game whose potential formulation takes the form of problem \eqref{pb:primal}. We describe it briefly in this subsection. Consider a finite set of nodes $\mathcal{N}$ and a finite set of edges $\mathcal{E} \subseteq \mathcal{N} \times \mathcal{N}$. The model is a static model that describes how the agents move on the network, taking into account their origins and destinations as well as the congestion on each arc.

We fix a subset $X$ of $\mathcal{N} \times \mathcal{N}$. Each parameter $x=(x_1,x_2) \in X$ represents an origin-destination pair. Next, we denote by $Y$ the set of subsets of $\mathcal{E}$.
For each $x$, we consider a set of possible paths connecting $x_1$ and $x_2$, denoted $Z_x$. Mathematically, we simply describe a path as a subset of $\mathcal{E}$, so $Z_x \subseteq Y$.

For the definition of the potential problem, we define $\mathcal{H}= \R^{\mathcal{E}}$. The function $g$ is defined by $g \colon (x,y) \in Z \mapsto (g(x,y)_e)_{e \in \mathcal{E}} \in \mathcal{H}$, where
\begin{equation*}
g(x,y)_e = \begin{cases}
\begin{array}{cl}
1, & \text{ if $e \in y$}, \\
0, & \text{ otherwise.}
\end{array}
\end{cases}
\end{equation*}
In words, $g(x,y)_e = 1$ is the edge $e$ belongs to the path $y$, 0 otherwise. Next, we fix a family of functions $\phi_e \colon [0,\infty) \rightarrow [0,\infty)$, parametrized by $e \in \mathcal{E}$. We assume that these functions are non-decreasing and we fix a primitive $\Phi_e$ for each of them. The functions $\Phi_e$ are convex. Finally, we define $f \colon \mathcal{H} \rightarrow \R$ by
\begin{equation*}
f(q) = \sum_{e \in \mathcal{E}} \Phi_e(q_e).
\end{equation*}

With these definitions at hand, it remains to interpret the optimality conditions for the associated MFO problem. Let us consider $\bar{\mu} \in \mathcal{P}_m(Z)$ and let $\bar{\lambda}= \nabla f( \int_Z g d \bar{\mu}) \in \mathcal{H}$. A direct calculation shows that
\begin{equation*}
\bar{\lambda}_e
= \phi_e ( {q}_e ), \quad
\text{where: } q_e= \sum_{\begin{subarray}{c}(x,y) \in Z \\ e \in y \end{subarray}}\bar{\mu}(x,y).
\end{equation*}
We can interpret $q_e$ as the proportion of agents using the edge $e$. We interpret $\phi_e$ as a function that gives the travelling time on the edge $e$ in function of the congestion $q_e$. So here the dual variable $\bar{\lambda}$ has a natural interpretation as a vector containing all the travelling times of the network. Finally, for any $x \in X$, we have
\begin{equation*}
\inf_{y \in Z_x} \ \langle \bar{\lambda}, g(x,y) \rangle
= \inf_{y \in Z_x} \ \sum_{e \in y} \bar{\lambda}_e.
\end{equation*}
Here $\sum_{e \in y} \bar{\lambda}_e$ describes the total duration of the path $y$. As a consequence, solving the MFO problem is equivalent to find $\bar{\mu}$ such that for any $x$, $\bar{\mu}_x$ is supported by the optimal paths (among those connecting $x_1$ to $x_2$), the travel time of a path $y$ begin defined by the above relations. This notion of equilibrium is known as Wardrop equilibrium in the literature. Using the MFO setting, we recover the well-known equivalence between Wardrop equilbria and their potential formulation, see \cite[Chapter 3]{sheffi1985urban}. We mention here that the modelling is different (but equivalent) to the standard one in which one rather describes the distribution of the agents with respect to the edges, instead of using the distribution with respect to the paths. We also note that the Frank-Wolfe algorithm is a very standard algorithm for solving those problems, see \cite[Section 5.2]{sheffi1985urban}.

\subsection{Lagrangian MFGs}

We propose here a class of potential Lagrangian MFGs. As before, we first formulate the potential problem, in the form of an MFO problem and interpret next the optimality conditions as a non-atomic game.
There is a large amount of literature on Lagrangian MFGs, we refer the reader to
\cite{benamou2017variational,bonnans2023lagrangian,cannarsa2018existence,santambrogio2021cucker,sarrazin2022lagrangian} and the references therein. Our intention here is only to formulate a model that fits with the framework of MFO problems, we do not check the corresponding assumptions, which must be done on a case-by-case basis.

Let us fix a domain $\Omega\subseteq\R^d$ and a final time $T > 0$. Let $\textnormal{AC}([0,T],\R^d)$ be the set of all absolutely continuous functions from $[0,1]$ to $\R^d$. For any $x \in \Omega$, we denote,
\begin{equation*}
Y \coloneqq \{ y \in \textnormal{AC}([0,T],\R^d) \,\mid \, y(t)\in \Omega, \, \forall t\in[0,T] \}, \qquad Z_x \coloneqq \{ y \in Y \, \mid \, y(0)=x \}.
\end{equation*}
Let $Z = \left\{ (x,y) \, \mid \, x\in \Omega, \, y \in Z_x \right\}$.
Let $m \in \mathcal{P}(\Omega)$ be the distribution of the initial states of the players. 
We fix three functions $L \colon \R^d \rightarrow \R$, $h \colon \R^d \rightarrow \R^k$, and $\Phi \colon \R^k \rightarrow \R$. We define $g \colon (x,y) \in Z \rightarrow (g_1(x,y),g_2(x,y)) \in \R \times L^2(0,T;\R^k)$ by
\begin{equation*}
g_1(x,y)= \int_0^T L(\dot{y}(t)) dt
\quad \text{and} \quad
g_2(x,y) = h \circ y
\end{equation*}
and we define $f \colon \R \times L^2(0,T;\R^k) \rightarrow \R$ as
\begin{equation*}
f(q_1,q_2)
= q_1 + \int_0^T \Phi(q_2(t)) dt.
\end{equation*}
With these definitions of $X$, $Y$, $Z_x$, $g$, and $f$, we have a full description of an MFO problem.
Let us write the corresponding optimality conditions. Let $\bar{\mu} \in \mathcal{P}_m(Z)$ and let $\bar{\lambda}= \nabla f(\int_Z g d \bar{\mu})$. Then, $\bar{\lambda}
= (\bar{\lambda}_1,\bar{\lambda}_2)$, with $\bar{\lambda}_1= 1$ and
\begin{equation*}
\bar{\lambda}_2(t)
= \nabla \Phi \Big( \int_Z h(y(t)) d \bar{\mu}(x,y) \Big).
\end{equation*}
For any $t \in [0,T]$, denote by $e_t \colon Z \rightarrow \R^d$ the mapping defined by $e_t(x,y)= y(t)$. Then $\bar{\lambda}_2$ is equivalently defined by
\begin{equation*}
\bar{\lambda}_2(t)
= \nabla \Phi \Big( \int_{\Omega} h(y') d \bar{m}_t(y') \Big), \quad
\text{where: } m_t= e_t \sharp \bar{\mu}.
\end{equation*}
In this context, the minimization problem in \eqref{pb:fix} is equivalent to the following optimal control problem:
\begin{equation*}
\inf_{y \in Z_x} \
\int_0^T
\Big[ L(\dot{y}(t))
+ \Big\langle \nabla \Phi \big( {\textstyle \int_{\Omega}} h d \bar{m}_t \big), h(y(t)) \Big\rangle \Big] dt.
\end{equation*}
Let us note that the above problem is not convex in general. It could be solved by dynamic programming in the situation where the dimension of $y$ is moderate.

\subsection{Numerical results for a competition problem with a non-renewable resource}

\subsubsection*{Model}
We consider a Lagrangian MFG in which the agents exploit their own stock of an exhaustible resource. The model is taken from \cite{graewe2022maximum}. We fix a time horizon $[0,T]$ where $T\in [0,+\infty)$ (the case $T=\infty$ investigated in \cite{graewe2022maximum} is not considered here).
The state variable of a representative agent is the level of the stock of resource at any time, denoted $(X_t^q)_{t \in [0,T]}$ and the control is the speed of extraction at any time, denoted $q$.
The dynamic of a given producer with an initial position $x_0\geq 0$ is described as follows:
\begin{equation*}
X_{t}^q \coloneqq x_0 - \int_{0}^t q_{\tau} d\tau, \qquad t\in [0,T],
\end{equation*}
where $q_t \geq 0$, for any $t \in [0,T]$.
We impose that $X_T^q \geq 0$, which implies that $X_t^q \geq 0$ at any time.

We define the set of aggregate production, denoted as $\mathcal{G}$, by
\begin{equation*}
\mathcal{G}\coloneqq \Big\{ Q \in \mathbb{L}^2([0,T],\R) \,\mid\, 0\leq Q(t) \leq \frac{1}{2}, \, \forall t\in [0,T]\Big\}.
\end{equation*}
The price of the resource for this representative producer depends on its extracting speed and an aggregate production $Q\in \mathcal{G}$,
\begin{equation*}
p_t \coloneqq 1 - q_t - \epsilon Q_t, \qquad t\in [0,T],
\end{equation*}
where $\epsilon \in (0,1)$ is a constant. The gain of this representative producer writes,
\begin{equation*}
\int_{0}^T e^{-rt} q_t(1-q_t-\epsilon Q_t) dt.
\end{equation*}
where $r\geq 0$ is a discount rate.
Therefore, given an aggregate production $Q\in \mathcal{G}$ and an initial position $x_0\geq 0$, we can formulate an optimal control problem associated with this representative producer,
\begin{equation}\label{pb:OC}
\begin{cases}
\begin{array}{rl}
{\ \displaystyle \inf_{q\in \mathcal{G}}} & J^{Q}(q) \coloneqq {\displaystyle \int_{0}^T} e^{-rt} q_t(q_t-1+\epsilon Q_t) dt; \\[1.0em]
\ \text{s.t.} &{\displaystyle \int_{0}^T} q_t dt \leq x_0.
\end{array}
\end{cases}
\end{equation}

\begin{lem}\label{lm:existence_uniqueness}
Problem \eqref{pb:OC} has a unique solution $q^{Q}(x_0)$. Moreover, $0 \leq q^Q(x_0)(t) \leq \frac{1}{2}$, for a.e.\@ $t \in (0,T)$.
\end{lem}

\begin{proof}
It is easy to see that $\mathcal{G}$ is a non-empty and convex subset of $\mathbb{L}^2([0,T],\R)$. Following \cite[Thm.\@ 3.12]{rudin}, if $(f_n\in \mathcal{G})_{n\geq 1}$ converges to $f$ in $\mathbb{L}^2$ sense, then there exists a subsequence of $(f_n)_{n\geq 1}$ converges to $f$ a.e. As a consequence, $f$ lies in $\mathcal{G}$. Therefore, $\mathcal{G}$ is closed. Furthermore, by H\"older's inequality, we obtain that $ \{q\in \mathbb{L}^2([0,T],\R) \, \mid \, \int_{0}^T q_t dt \leq x_0\}$ is non-empty, convex and closed in $\mathbb{L}^2([0,T],\R)$. It follows that the admissible set of problem \eqref{pb:OC} is non-empty, closed and convex in Hilbert space $\mathbb{L}^2([0,T],\R)$. On the other hand, the cost function $J^{Q}(\cdot)$ is strongly convex. Then the existence of the solution of \eqref{pb:OC} comes from \cite[Cor.\@ 3.23]{brezis2011functional} and the uniqueness is by the strong convexity of $J^{Q}$.

Let $q$ be the solution to \eqref{pb:OC}. Define $q'(t)= \min \{ q(t), \frac{1}{2} \}$, for a.e.\@ $t \in (0,T)$. Since $q' \leq q$, $q'$ is also feasible for problem \eqref{pb:OC}. Moreover, the running cost $q_0 \mapsto q_0(q_0-1 + \varepsilon Q_t)$ is increasing for $q_0 \geq \frac{1}{2}$. As a consequence, $J^Q(q') \leq J^Q(q)$. Therefore, $q'$ is optimal, and since the solution is unique, we have $q= q'$, which proves that $q \leq \frac{1}{2}$.
\end{proof}

Let $m \in \mathcal{P}([0,+\infty))$ denote the distribution of the initial conditions of the producers.
The aggregate production rate corresponding to $ q^{Q}$ is given by
\begin{equation*}
Q^Q_t \coloneqq \int_{0}^{\infty} q_t^{Q} (x_0)d m(x_0), \qquad \forall t\in [0,T].
\end{equation*}
Following \cite{graewe2022maximum}, we call Nash equilibrium a solution $Q^*$ to the fix-point problem:
\begin{equation}\label{pb:fixed_point}
Q^{*} = Q^{Q^{*}}, \qquad Q^{*}\in \mathcal{G}.
\end{equation}

\subsubsection*{Potential problem}

In this paragraph, we find an optimization problem associated with the fixed point problem \eqref{pb:fixed_point}, which is a particular case of problem \eqref{pb:primal}. Let us specific  metric spaces and admissible sets in \eqref{pb:primal} associated with \eqref{pb:fixed_point}:
\begin{equation*}
X = [0,\infty), \qquad Y = \mathcal{G}, \qquad F(x) = \left\{q\in\mathcal{G}\,\mid\, \int_{0}^T q_t dt \leq x \right\},\qquad Z= \text{Graph}(F), \qquad Z_x = F(x).
\end{equation*}
Let us define the separable Hilbert space $\mathbb{L}^2_{e^{-rt}}([0,T])$ \cite[Example.\@ 4.5(b)]{rudin}:
\begin{equation*}
\mathbb{L}^2_{e^{-rt}}([0,T])\coloneqq \left\{ \zeta\colon [0,T]\to \R \text{ is Lebesgue measurable}  \,\Big|\,  \int_{0}^T e^{-rt}|\zeta(t)|^2dt <+\infty \right\},
\end{equation*}
with a scalar product,
\begin{equation*}
\langle f_1, f_2\rangle_{\mathbb{L}^2_{e^{-rt}}([0,T])} = \int_{0}^T  e^{-rt} f_1(t)f_2(t) dt.
\end{equation*}
It is easy to check that $Y=\mathcal{G}\subseteq  \mathbb{L}^2_{e^{-rt}}([0,T])$. Then, in \eqref{pb:primal}, we set $\mathcal{H} =  \R\times \mathbb{L}^2_{e^{-rt}}([0,T])$,
\begin{align*}
&g \colon Z\to \mathcal{H}, (x,q) \mapsto \left( \int_{0}^T e^{-rt} (q_t^2-q_t) dt, \, q \right), \\
&f\colon \mathcal{H}\to \R,  (y_1,y_2) \mapsto  y_1 + \frac{\epsilon}{2}\|y_2\|_{\mathbb{L}^2_{e^{-rt}}([0,T]) }^2.
\end{align*}
Therefore, problem \eqref{pb:primal} associated with \eqref{pb:fixed_point} writes:
\begin{equation}\label{pb:Potential}
\inf_{\mu\in \mathcal{P}_{m}(Z)} \int_{Z}\int_{0}^T e^{-rt}(q_t^2-q_t)dt d\mu(x,q) + \frac{\epsilon}{2} \int_{0}^T e^{-rt} \left(\int_{Z} q_t d\mu(x,q)\right)^2 dt.
\end{equation}


\begin{prop} \label{prop:potential}
If $\bar{\mu}$ is a solution of problem \eqref{pb:Potential}, then $Q^{*} = \int_{Z} q d\bar{\mu}(x,q)$ is a Nash equilibrium of the optimal exploitation of exhaustible resources problem, i.e., $Q^{*}$ is a solution of \eqref{pb:fixed_point}.
\end{prop}

\begin{proof}
Let us first check that Assumption \ref{ass1} holds true for problem \eqref{pb:Potential}.
It is easy to see that Assumption \ref{ass1}(1) and the first and the third points in Assumption \ref{ass1}(2) are true by the continuity of $g$ and Lemma \ref{lm:existence_uniqueness}. Let us prove that $G_{\lambda}$ is lower semi-continuous for any $\lambda\in \mathcal{H}_f$. This is a consequence of the claim that the set-valued function $\mathcal{Z}\colon X \rightsquigarrow \mathcal{H}$,
$x\mapsto \{g(x,y) \mid y\in Z_x\}$ is locally Lipschitz, i.e. Lipschitz in any compact set of $X$.
To see the local Lipschitz continuity, we fix any $x_1<x_2$ in $X$. If $q\in Z_{x_1}$, then we have immediately that $q\in Z_{x_2}$. This implies that $Z_{x_1}\subseteq Z_{x_2}$. On the other hand, let $q\in Z_{x_2}$. We construct $q'\in Z_{x_1}$ by the following method:
\begin{equation*}
q'_t = \begin{cases}
q_{t},\quad &\text{if } \int_{0}^t q_{\tau} d\tau \leq x_1;\\
0 , & \text{otherwise}.
\end{cases}
\end{equation*}
As a consequence, we have that  $\|q'-q\|_{\mathbb{L}^{1}([0,T])} \leq x_2-x_1$. Therefore, by H\"older's inequality,
\begin{equation*}
\|q'-q\|_{\mathbb{L}^2_{e^{-rt}}([0,T]) }^2 \leq \|e^{-rt}(q'-q)\|_{\mathbb{L}^{\infty}([0,T])} \|q'-q\|_{\mathbb{L}^{1}([0,T])}  \leq x_2(x_2-x_1).
\end{equation*}
This implies that $Z_{x_2} \subseteq Z_{x_1} + \mathcal{B}_{Y}(0, \sqrt{x_2(x_2-x_1)})$. Therefore, Assumption \ref{ass1} follows.

Let $\bar{\mu}$ be a solution of problem \eqref{pb:Potential}, $\bar{\lambda} = \nabla f(\int_{Z}g d\bar{\mu})$ and $Q^{*}=\int_{Z}q d\bar{\mu}(x,q)$. By the definitions of $f$ and $g$, we obtain that $\bar{\lambda} = (1,\epsilon Q^{*})$, moreover,
\begin{equation*}
g_{\bar{\lambda}}(x,q) = \int_{0}^T e^{-rt} q_t(q_t-1 + \epsilon Q^{*}_t) dt.
\end{equation*}
By Lemma \ref{lm:existence_uniqueness}, $\br_{\bar{\lambda}}(x_0)  = \{q^{Q^{*}}(x_0)\}$ for any $x_0\in X$. By Corollary \ref{cor:fixed_point}, we have that $(\bar{\lambda},\bar{\mu})$ satisfies the following equilibrium equation:
\begin{equation*}
\begin{cases}
\ \bar{\lambda} = \left(1\, ,\, \epsilon \int_{Z} q d \bar{\mu}\right)\\
\ \bar{\mu}_{x} = \delta_{q^{Q^{*}}(x)}, \quad m\text{-a.e.}
\end{cases}
\end{equation*}
Combining with Theorem \ref{thm:disintegration}, we obtain that $\int_{Z} q d \bar{\mu} = \int_{X}q^{Q^{*}}(x)dm(x) $. Recall that $Q^{*} = \int_{Z}q d\bar{\mu}$, then \eqref{pb:fixed_point} follows.
\end{proof}

\subsubsection*{Numerical simulations}

Let the initial measure $m$ be an exponential distribution with parameter $a\geq 0$, i.e., $dm(x) = a e^{-ax} dx$ for all $x\geq 0$.
Let us independently sample the distribution $m$ for $N$ times, denoting the samples by $x_1, x_2, ..., x_N$, and $m_N = \frac{1}{N}\sum_{i=1}^N \delta_{x_i}$.
The time space $[0,T]$ is discretized with a step size $\Delta t = T/M$ for some $M\geq 1$.
Then, a totally discretized problem associated with \eqref{pb:Potential} writes:
\begin{equation}\label{pb:Potential_dis}
\begin{cases}
\ \inf_{q\in \R^{N\otimes M} }& J_N(q)\coloneqq
\frac{\Delta t}{N}  \sum_{i=1}^N \sum_{t=0}^{M-1} e^{-rt\Delta t}(q_{i,t}^2-q_{i,t}) + \frac{\epsilon \Delta t}{2} \sum_{t=0}^{M-1} e^{-rt} \left( \frac{1}{N}\sum_{i=1}^{N}{q_{i,t}} \right)^2, \\
\ \text{such that } & q_i\in  S^{M}(x_i) \coloneqq \{q\in [0,1/2]^M \mid \Delta t \sum_{t=0}^{M-1} q_t \leq x_i  \}, \qquad i =1,2,\ldots, N.
\end{cases}
\end{equation}

We apply Algorithm \ref{alg:SFW} to solve \eqref{pb:Potential_dis}. At each iteration, the evaluation of a best-response, for each producer $i$ amounts to solve a problem of the following form:
\begin{equation}\label{pb:sub_i}
\begin{cases}
\ \inf_{q_i\in \R^M} & \Delta t  \sum_{t=0}^{M-1} e^{-rt\Delta t}q_{i,t}(q_{i,t}-1 + \epsilon Q_t ), \\
\ \text{such that } & q_i\in  S^{M}(x_i),
\end{cases}
\end{equation}
for a given $Q\in [0,1/2]^M$. This problem is a convex quadratic programming problem in $\R^{M}$ that can be dealt with by some solvers, such as GUROBI \cite{gurobi2018gurobi}.

For the resolution of the problem, we chose the following parameters: $T=10$, $\epsilon = r=a=1$, $N=100$, $M = 100$, $K= 100$, $n_k = 10$, for all $k$.
Figure \ref{fig:SS} shows the extracting speeds and the stocks of three producers with initial stocks: 0.9, 1.2, and 3.1.
From Figure \ref{fig:SS}, we see that the producers with the higher initial stock have the same extracting speed as those with a lower initial stock, at the beginning. However, as the smaller agents exhaust their resource, the larger ones progressively raise their extraction speed. Once the extraction speed reaches its maximum value, it rapidly decreases to zero.
These observations are consistent with the findings of \cite[Sec.\@ 3.3]{graewe2022maximum}.

\begin{figure}[ht]
\centering
\includegraphics[width=1\linewidth]{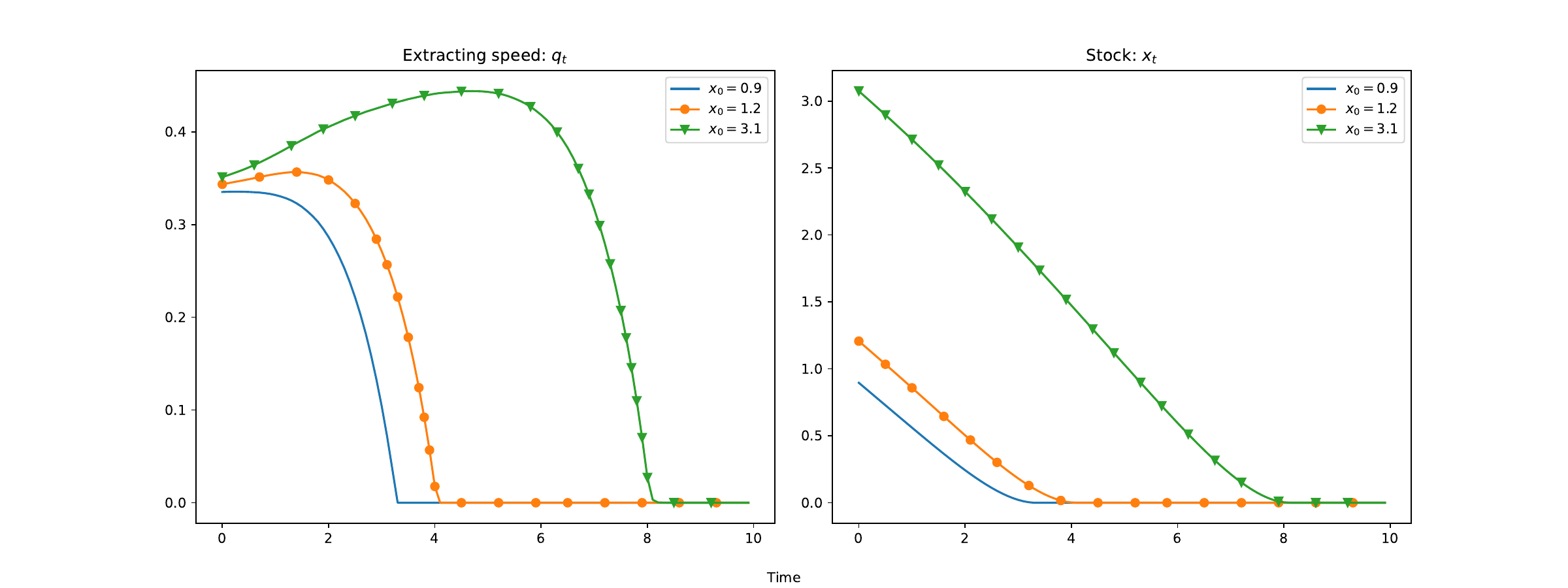}
\caption{Extracting speeds and stocks of three producers with initial stocks: 0.9, 1.2, and 3.1.}
\label{fig:SS}
\end{figure}

To study the error caused by sampling, we independently sample the exponential distribution $m$ for $100*N$ times, and group them into batches of $N$. The empirical distribution corresponding to each batch is set as the initial distribution.
Then we apply Algorithm \ref{alg:SFW} to compute $Q^{*}$ corresponding to each initial distribution. In Figure \ref{fig:MS}, we show the mean and standard deviation of the results of the 100 simulations.

\begin{figure}[ht]
\centering
\includegraphics[width=0.5\linewidth]{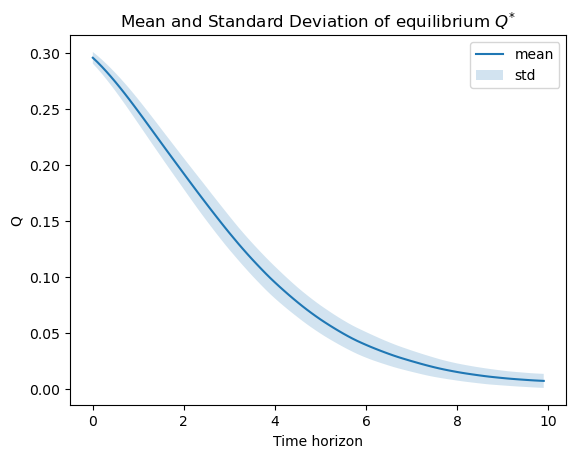}
\caption{Mean and standard deviation of the equilibria of $100$ batches}
\label{fig:MS}
\end{figure}

\subsection{Numerical results for a congestion game}

\subsubsection*{Model}

Consider a second numerical example within the context of the minimal-time deterministic MFG. We set the following parameters: the state space is fixed as $[0, 1]$, a maximum duration is denoted by $T>0$, and an upper bound for the speed is given by $\bar{V}>0$. In this particular example, the dynamics governing each player are characterized by the set $Z$, defined as:
\begin{equation*}
Z = \left\{ (x,\gamma)\in [0,1]\times \text{AC}([0,T]) \, \mid \, \dot{\gamma}_t = v_t, \,\gamma_0 = x, \,0\leq v_t\leq \bar{V} \right\}.
\end{equation*}
The objective for the players in this example is to reach the target point $1$ as soon as possible, while simultaneously ensuring that the density at each point does not become excessively high. To quantify this, we introduce the congestion function $\tilde{\mathcal{F}}\colon \mathcal{P}([0,1])\to \R$, which is defined as follows:
\begin{equation*}
\tilde{\mathcal{F}}(m) = \begin{cases}
\int_{0}^1 m(x)^2 dx,\quad &\text{if } m\ll dx, \\
+\infty, \text{otherwise}.
\end{cases}
\end{equation*}
Given an initial distribution $m_0\in \mathcal{P}([0,1])$, the resulting deterministic MFG problem can be expressed as follows:
\begin{equation}\label{eq:1stmfg}
\inf_{\mu\in \mathcal{P}_{m_0}(Z)} J(\mu) \coloneqq
\int_{Z}\int_{0}^T \mathbb{I}_{[0,1]}(\gamma_t) dt d\mu(x,\gamma) + \alpha \int_{0}^T \tilde{\mathcal{F}} (e_t\#\pi_2\#\mu) dt,
\end{equation}
where $\alpha>0$ is a penalty parameter.

\subsubsection*{Regularization}

Note that the congestion function $\tilde{\mathcal{F}}$ does not fit to the framework studied in this article. To address this, we begin by approximating $\tilde{\mathcal{F}}$ with a function that aligns with our framework. We achieve this by partitioning the interval $[0,1]$ into $J\in \mathbb{N}_{+}$ small, uniform subintervals: $I_1,\ldots, I_J$, where $I_j = [(j-1)\Delta x, j\Delta x]$, and $\Delta x = 1/J$. Subsequently, we approximate $\tilde{\mathcal{F}}(m)$ as follows:
\begin{equation*}
\tilde{\mathcal{F}}(m) \approx  \Delta x\sum_{j=1}^J \left(\frac{1}{\Delta x}\int_{0}^1 \mathbb{I}_{I_j}(x)dm(x)\right)^2.
\end{equation*}
To facilitate the execution of the numerical experiments, we replace the indicator function $\mathbb{I}_{I_j}$ by some smooth functions. Let $k\geq J$ be a positive integer. We introduce two smooth functions, denoted as $\varphi_k \in \mathcal{C}^{\infty}(\mathbb{R})$ and $\phi_{k,\Delta x} \in \mathcal{C}^{\infty}(\mathbb{R})$. These functions are parametrized by the variables $k$ and $\Delta x$, and are defined as follows:
\begin{equation*}
\begin{split}
\varphi_{k} (x) &=
\begin{cases}
\ 0 , & \text{if }  x<=0,\\
\ \frac{1}{1+ e^{1/kx-1/(1-kx)}}, \qquad &\text{if } 0<x < \frac{1}{k},\\
\ 1 & \text{otherwise},
\end{cases} \\[0.8em]
\phi_{k,\Delta x} (x) & =
\begin{cases}
\ 0 , & \text{if }x\leq -\frac{1}{k},\\
\ \varphi_{k} (x+1/k) , \qquad &\text{if } -\frac{1}{k}<x<0, \\
\ 1 &\text{if } 0\leq x \leq \Delta x-\frac{1}{k}, \\
\ 1- \varphi_{k} (x-\Delta x+\frac{1}{k}), &\text{otherwise}.
\end{cases}
\end{split}
\end{equation*}
Then, we approximate $\mathbb{I}_{I_j}$ by $h_j\colon \mathbb{R}_{+}\to [0,1], x\mapsto \phi_{k,\Delta x}(x-(j-1)\Delta x)$ for $j=1,\ldots, J$, and $\mathbb{I}_{[0,1]}$ by $h_0\colon \mathbb{R}_{+}\to [0,1]$,
\begin{equation*}
h_0 (x) = \begin{cases}
1, \qquad &\text{if } 0\leq x < 1-\frac{1}{k},\\
\ 1-\varphi_k\left(x-1+\frac{1}{k}\right), \quad   & \text{otherwise}.
\end{cases}
\end{equation*}
An important property of $\{h_j\}_{j=1,\ldots,J}$ is that $\sum_{j=1}^J h_j(x) =h_0(x)$, which is $1$ for any $x\in [0,1-1/k]$, see Figure \ref{fig:h}.
\begin{figure}[ht]
\centering
\includegraphics[width=1\linewidth]{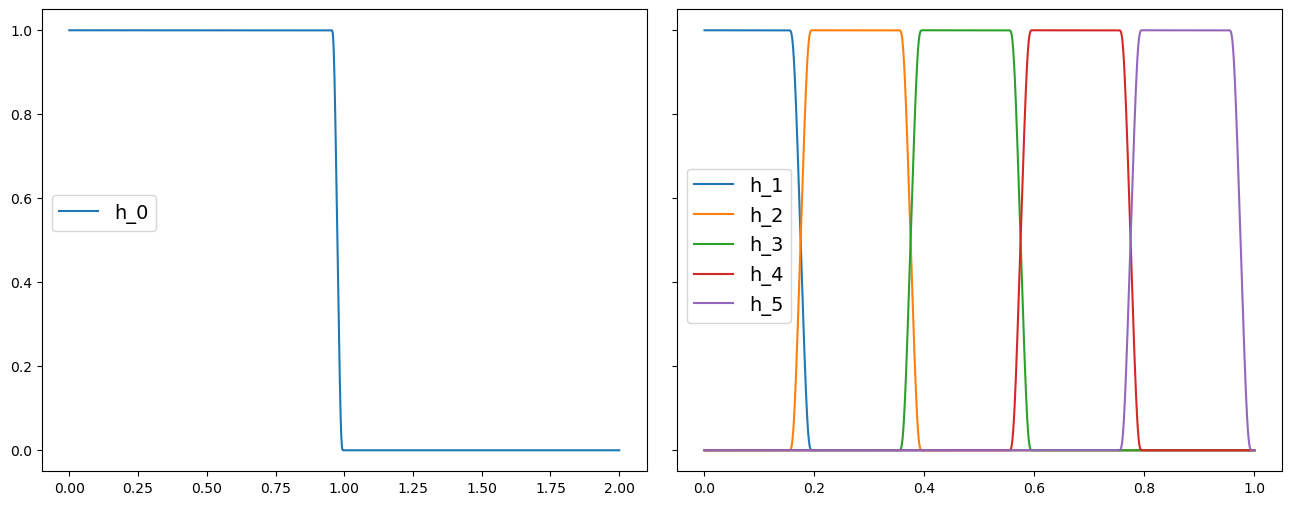}
\caption{Example of  $h_0$ and $\{h_j\}_{j=1,\ldots,J}$ with $J=5$ and $k=20$.}
\label{fig:h}
\end{figure}
The resulting approximated MFO problem associated with \eqref{eq:1stmfg} is,
\begin{equation}\label{pb:1stmfg_approx}
\inf_{\mu\in \mathcal{P}_{m_0}(Z)}
\int_{Z}\int_{0}^T h_0(\gamma_t) dt d\mu(x,\gamma) + \frac{\alpha}{\Delta x} \sum_{j=1}^J \int_{0}^T \left(\int_{Z} h_j(\gamma_t) d\mu(x,\gamma)\right)^2 dt.
\end{equation}

\subsubsection*{Discretization and numerical results}

To proceed with our numerical experiments, we discretize the time horizon $[0, T]$ into $M$ steps, each of duration $\Delta t = T/M$. Additionally, we discretize the initial distribution $m_0$ by $m_N = \frac{1}{N}\sum_{i=1}^N\delta_{x_i}$. We formulate the fully discretized problem associated with \eqref{pb:1stmfg_approx} as follows:
\begin{equation}\label{pb:1stmfg_dis}
\begin{cases}
\inf_{(\gamma^i)_{i=1,\ldots, N}} \quad &
\frac{\Delta t}{N}\sum_{i=1}^N \sum_{t=0}^{M-1} h_0(\gamma^i_t) + \frac{\alpha\Delta t}{\Delta x} \sum_{j=1}^J \sum_{t=0}^{M-1} \left(\frac{1}{N} \sum_{i=1}^N h_j(\gamma_t^i) )\right)^2,\\
\text{such that } & 0\leq \gamma_{t+1}^i -\gamma_t^i \leq \bar{V}\Delta t, \, \gamma_0^i =x_i, \, \, \text{for } t= 0,\ldots, M, \text{ and }i=1,\ldots,N.
\end{cases}
\end{equation}
Therefore, given some $(\bar{\gamma}^i)_{i=1,\ldots,N}$ satisfying the constraint in \eqref{pb:1stmfg_dis}, the sub-problem for player $i$ is
\begin{equation}\label{pb:1stmfg_sub}
\begin{cases}
\inf_{\gamma^i} \quad &
\Delta t \sum_{t=0}^{M-1} h_0(\gamma^i_t) + \frac{2\alpha\Delta t}{\Delta x} \sum_{j=1}^J \sum_{t=0}^{M-1} \bar{y}_{j,t} h_j(\gamma_t^i) ,\\
\text{such that } & 0\leq \gamma_{t+1}^i -\gamma_t^i \leq \bar{V}\Delta t, \, \gamma_0^i =x_i, \, \, \text{for } t= 0,\ldots, M,
\end{cases}
\end{equation}
where $\bar{y}_{j,t} = \frac{1}{N}\sum_{i=1}^N h_j(\bar{\gamma}^i_t)$. The sub-problem \eqref{pb:1stmfg_sub} is a finite-dimensional non-convex optimization problem, which is addressed by the open-source solver ``scipy.optimize.minimize" \cite{virtanen2020scipy}.

Let us specify the parameters used in the numerical simulation of problem \eqref{pb:1stmfg_approx} as follows:
\begin{equation*}
J = 5,\, k=20,\, \bar{V}=3,\,   T=1,\,  \alpha =1, \, M = 20, \, N = 200 ,\, m_0 = \text{Uni}_{[0,0.2]},
\end{equation*}
where ``Uni" represents the uniform distribution, and the points $x_i$ are drawn from samples of $m_0$.

We first present in Figure \ref{fig_conv} convergence results of Algorithm \ref{alg:SFW} for the discretized problem \eqref{pb:1stmfg_dis} in $100$ iterations, utilizing parameter settings of $n_k=1$ and $5$. We see that in both choices, the algorithm converges to a local minimum very fast (fewer than 20 iterations).
In Figure \ref{fig_traj}, we compare the optimal trajectories of $\gamma^i$ for two cases: $\alpha=0$ and $\alpha=1$. It is evident that, in the case of $\alpha=0$, the optimal strategy for each player is to move at the maximum speed, $\bar{V}=3$, as there is no penalty for density. This is depicted in the left part of Figure \ref{fig_traj}. However, when $\alpha=1$, players starting from greater initial positions choose to run at the maximum speed, whereas those with lower initial positions prefer to wait briefly to avoid congestion in density with those starting farther ahead. In Figure \ref{fig_m1}, we draw the agents' state distributions at each time for both $\alpha=0$ and $\alpha=1$, along with a regularized version obtained through interpolation in Figure \ref{fig_m2}. For a more detailed view of the density evolution before reaching the target, we further depict the restricted distributions within the spatial interval $[0, 0.8]$ in Figure \ref{fig_m3}, using a distinct color scale.

\begin{rem}
Let us underline that for this example, the optimization problems involved in the evaluation of the best-response mapping are non-convex. As mentionned above, we address them with 
the open-source solver scipy.optimize.minimize whose default method for tackling constrained non-linear optimization problems is the SLSQP (Sequential Least SQuares Programming) algorithm, a quasi-Newton-type algorithm. Consequently, the quality of the initial guess plays a crucial role in the resolution of sub-problems. In the context of Algorithm \ref{alg:SFW}, our experience shows that at iteration $k$, it is more efficient to initialise the evaluation of $\bar{y}_i^k \in \br_{\lambda^k}(x_i)$ with $y_i^{k}$ (rather than $\bar{y}_i^{k-1}$). We conjecture that the chance for the solver to generate a local solution is higher when initializing with $\bar{y}_i^{k-1}$.
\end{rem}

\begin{figure}[htbp]
\centering
\includegraphics[width=0.5\linewidth]{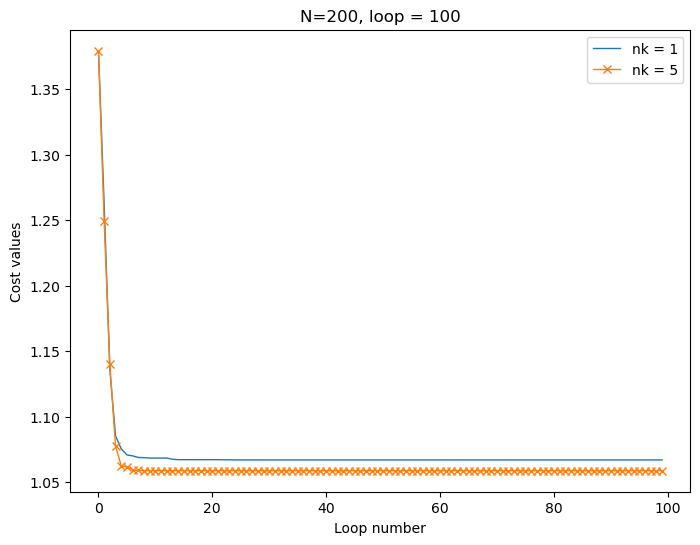}
\caption{Convergence results of Algorithm \ref{alg:SFW}.}
\label{fig_conv}
\end{figure}

\begin{figure}[htbp]
\centering
\includegraphics[width=0.8\linewidth]{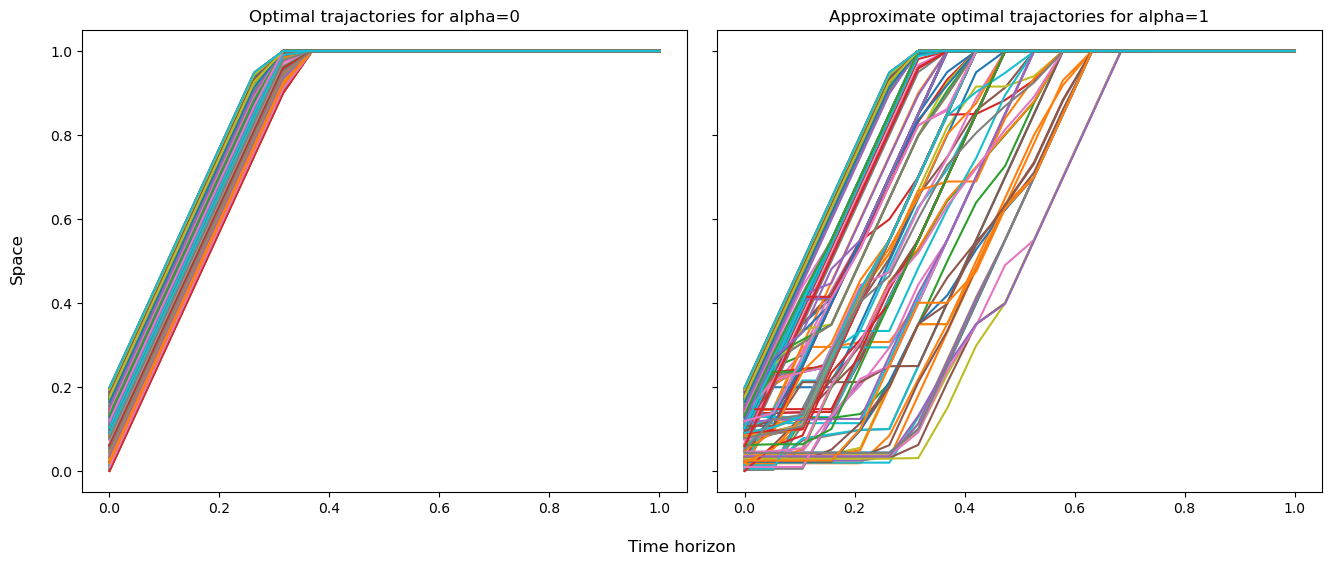}
\caption{Optimal trajectory of each player: the case $\alpha=0$ (left), the case $\alpha =1$ (right)}
\label{fig_traj}
\end{figure}

\begin{figure}[htbp]
\centering
\begin{subfigure}{0.8\textwidth}
\centering
\includegraphics[width=\linewidth]{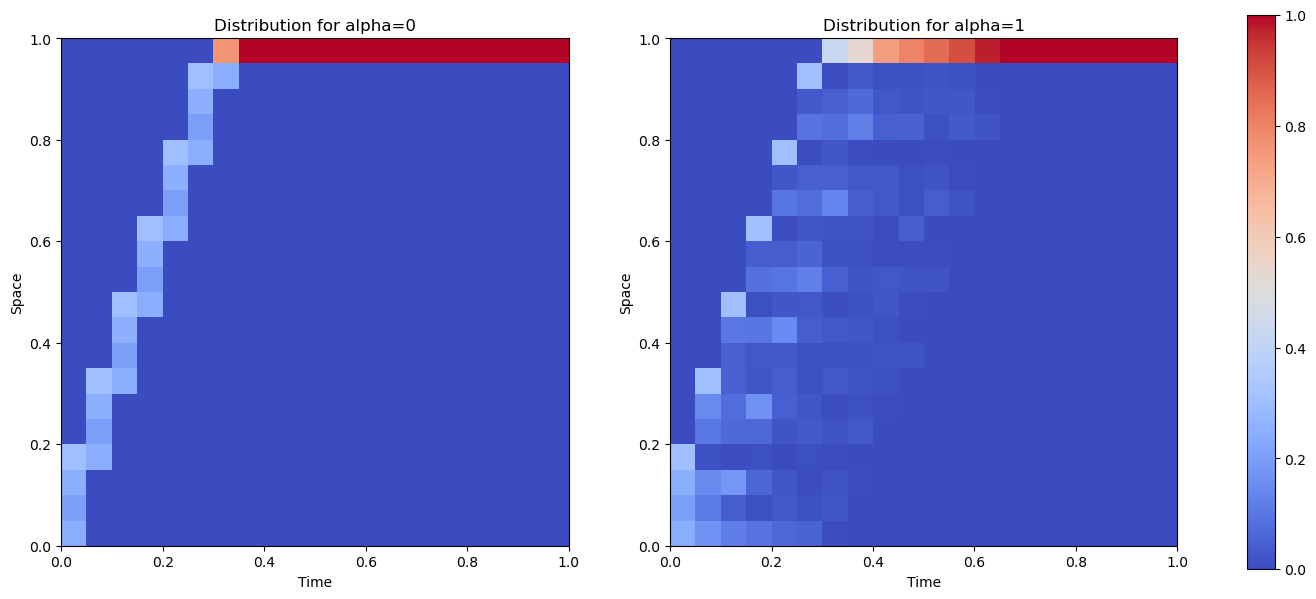}
\caption{Comparison of distributions of positions at each time: the case $\alpha=0$ (left), the case $\alpha=1$ (right).}
\label{fig_m1}
\end{subfigure}
\hfill

\begin{subfigure}{0.8\textwidth}
\centering
\includegraphics[width=\linewidth]{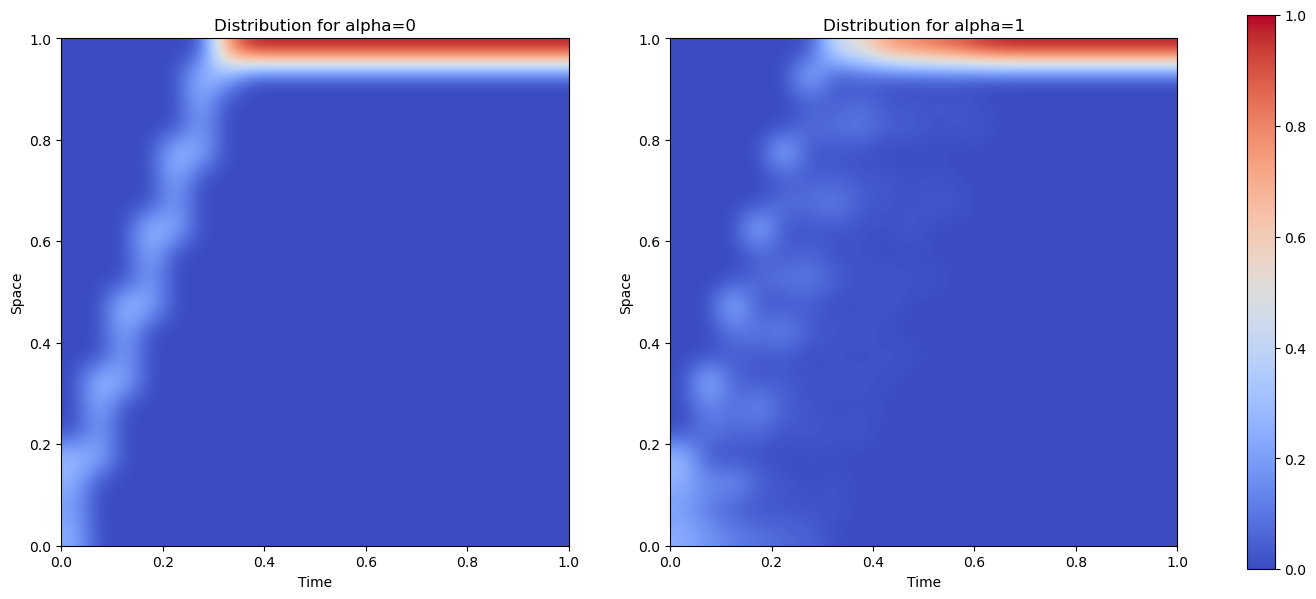}
\caption{Comparison of regularized distributions of positions at each time: the case $\alpha=0$ (left), the case $\alpha=1$ (right).}
\label{fig_m2}
\end{subfigure}
\hfill

\begin{subfigure}{0.8\textwidth}
\centering
\includegraphics[width=\linewidth]{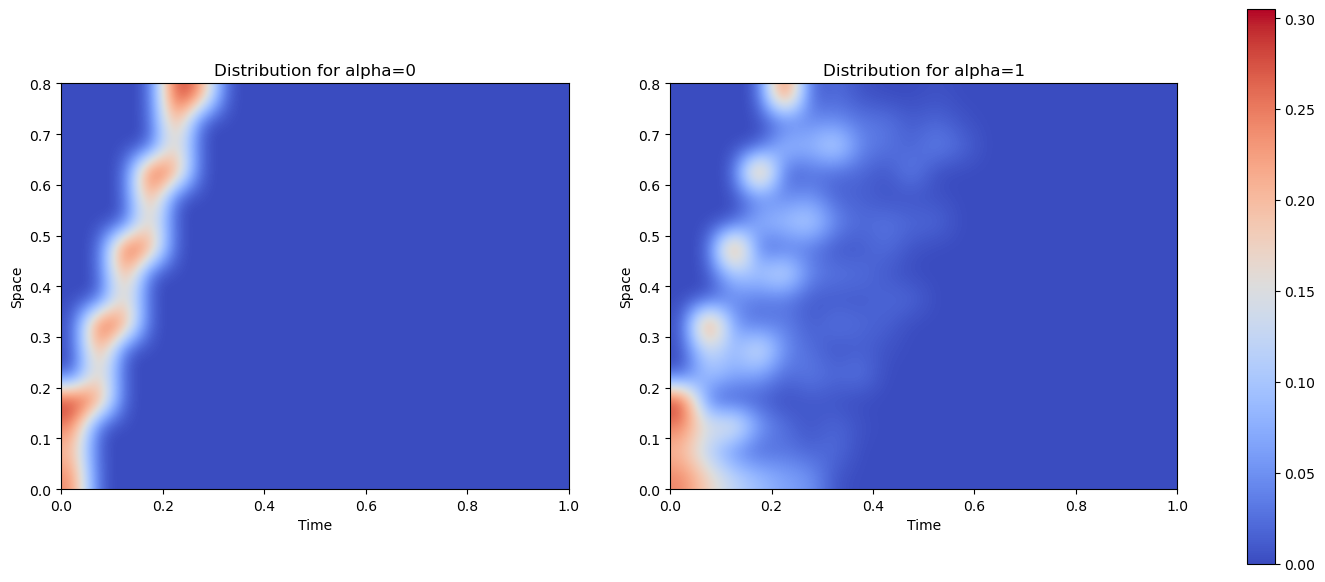}
\caption{Comparison of regularized distributions of positions in horizon $[0,0.8]$ at each time: the case $\alpha=0$ (left), the case $\alpha=1$ (right).}
\label{fig_m3}
\end{subfigure}
\medskip
\caption{Distributions}
\label{fig_m}
\end{figure}

\section{Conclusion}

We have provided a general framework for analyzing Mean Field Optimization problems. We have proposed a general method, based on an extension of the Frank-Wolfe algorithm for solving MFO problems, with a convergence guarantee, assuming that some best-response function can be efficiently computed (with a solver or with specific methods). Numerous extensions of the current setting could be considered. For example, one could formulate a stochastic setting with a random variable impacting all agents. In this setting the evaluation of $\lambda$ (in the SFW algorithm) may require to use Monte-Carlo approximations, adding a new source of error in the general algorithm. One may also realize a general convergence analysis that would take into account the need to discretize the sets $Z_x$ (in particular in the case of MFGs, where $Z_x$ is an infinite dimension set). Finally, at a purely numerical level, we could investigate variants of the proposed method in which the distribution $m$ is discretized progressively. This would reduce the number of subproblems to solve in the early iterations of the SFW algorithm.
We also mention that the SFW is robust in the following sense: at the end of iteration $k$, if $y^{k+1}$ is replaced by any other point yielding a reduction of the cost function, then the general convergence properties of the SFW algorithm are preserved. This fact could motivate the design of heuristic improvements on a case-by-case basis.

\appendix

\section{Proof of Lemma \ref{lm:inf_int}}\label{Appx:A}

Before proving Lemma \ref{lm:inf_int}, let us recall the definitions of the restriction of a measure and the completion of a probability space, taken from \cite[Thm.\@ 1.36]{rudin}.

\begin{defn}[Restriction]\label{def_res}
Let $X_1$ be a Polish space, let $\mathcal{X}$ and $\mathcal{X}'$ be two $\sigma$-algebras on $X_1$ such that $\mathcal{X}'\subseteq \mathcal{X}$, and let $\nu$ be a measure on $\mathcal{X}$.
The restriction measure of $\nu$ on $\mathcal{X}'$ is defined as follows:
\begin{equation*}
\nu|_{\mathcal{X}'}(A) \coloneqq  \nu(A),  \qquad \text{ for any } A \in \mathcal{X}'.
\end{equation*}
\end{defn}

\begin{defn}[Completion]\label{def_com}
Let $(X_1,\mathcal{B}^{X_1},\nu)$ be a probability space. Let $\mathcal{B}_{\nu}$ be the collection of all $E\subseteq X_1$ such that
there exists $A$ and $B$ in $\mathcal{B}^{X_1}$, $A\subseteq E \subseteq B$, and $\nu(B-A) = 0$. For such an $E$, we define a function $\hat{\nu}(E)$ as
\begin{equation*}
\hat{\nu} (E) = \nu(A).
\end{equation*}
Then $(X_1, \mathcal{B}_{\nu} , \hat{\nu})$ is a complete measure space. We say that  $(X_1, \mathcal{B}_{\nu} , \hat{\nu})$ is the completion of $(X_1, \mathcal{B}^{X_1}, \nu)$.
\end{defn}

\begin{proof}[Sketch of the proof of Lemma  \ref{lm:inf_int}]
The proof of the direction that the left-hand-side of \eqref{eq:first_order_2} is greater than the right-hand-side is the same as the proof for the case that $m\in \mathcal{P}_{\delta}(X)$.

Let us prove the converse inequality. Let $(X, \mathcal{B}_m, \hat{m})$ be the completion of the probability space $(X, \mathcal{B}^X,m)$. Fix any $\lambda\in \mathcal{H}_f$. By Assumption \ref{ass1}, the set-valued function $\br_{\lambda}\colon X \rightsquigarrow Y $ has non-empty closed images. By Lemma \ref{lm:close}, Graph$(\br_{\lambda})$ is closed in $X\times Y$, thus is a $\mathcal{B}_m \otimes \mathcal{B}^Y$-measurable set. By Lemma \ref{lm:set_measurable_2} and Theorem \ref{thm:measurable_selection}, the set-valued function $\textnormal{\br}_{\lambda}\colon X \rightsquigarrow Y$ is $(\mathcal{B}_m, \mathcal{B}^Y)$-measurable, and there exists a $(\mathcal{B}_m, \mathcal{B}^Y)$-measurable function $\textnormal{\sbr}_{\lambda} \colon X \to Y$ such that for any $x\in X$,
\begin{equation*}
\textnormal{\sbr}_{\lambda}(x) \in \textnormal{\br}_{\lambda} (x).
\end{equation*}
We define $\mathcal{A}\colon X \to Z $, $x\mapsto (x, \sbr_{\lambda}(x))$.
Since $\sbr_{\lambda}$ is $(\mathcal{B}_m, \mathcal{B}^Y)$-measurable, we have that $\mathcal{A}$ is $(\mathcal{B}_m, \mathcal{B}_m\otimes \mathcal{B}^{Y})$-measurable, see \cite[Lem.\@ 1.8]{kallenberg1997foundations}.
Let $\mathcal{B}^Z$ be the Borel $\sigma$-algebra on $Z$.
It is obvious that $\mathcal{B}^Z \subseteq \mathcal{B}^X\otimes \mathcal{B}^{Y} \subseteq \mathcal{B}_m\otimes \mathcal{B}^{Y}$. Let us take
\begin{equation*}
\tilde{\mu} = \mathcal{A} \# \hat{m} |_{\mathcal{B}^Z}.
\end{equation*}
Then $\tilde{\mu}$ is a positive Borel measure on $Z$. Moreover, we deduce from Definitions \ref{def_res}-\ref{def_com} that
\begin{equation*}
\tilde{\mu}(Z) = \mathcal{A}\#\hat{m}(Z) = \hat{m}(X) =m (X) =1.
\end{equation*}
Therefore, $\tilde{\mu}\in \mathcal{P}(Z)$.
Assume that the following two equalities hold true:
\begin{align}
\pi_1\# \tilde{\mu} &= m , \label{eq:proof1}\\
\int_{Z} g_{\lambda} d\tilde{\mu} &= \int_X g_{\lambda}\circ \mathcal{A}\,  dm. \label{eq:proof2}
\end{align}
By the definitions of $u_{\lambda}$ and $\mathcal{A}$,
\begin{equation}\label{eq:proof3}
g_{\lambda}\circ \mathcal{A}(x) = g_{\lambda} (x, \sbr_{\lambda}(x)) = \inf_{y \in Z_x} g_{\lambda} (x, y) = u_{\lambda}(x), \quad \forall x\in X.
\end{equation}
Combining \eqref{eq:proof1}-\eqref{eq:proof3}, we obtain that
\begin{equation*}
\inf_{\mu\in \mathcal{P}_{m}(Z)}  \int_{Z} g_{\lambda} d  \mu \leq \int_{Z} g_{\lambda} d \tilde{\mu} = \int_X g_{\lambda}\circ \mathcal{A}\,  dm = \int_{X} u_{\lambda} dm .
\end{equation*}
The conclusion follows.
\end{proof}

For completing the proof of Lemma \ref{lm:inf_int}, it remains to prove equalities \eqref{eq:proof1}-\eqref{eq:proof2}. They are deduced from Lemmas \ref{lm:measure}-\ref{lm:lsc_Borel}:
\begin{itemize}
\item To prove \eqref{eq:proof1}, we take $\tilde{X}= X$ and $h=\pi_1$ in Lemma \ref{lm:measure};
\item To prove \eqref{eq:proof2}, we take $\tilde{X} = [-M\|\lambda\|,+\infty)$ and $h = g_{\lambda}$ in Lemma \ref{lm:measure}, and Lemma \ref{lm:lsc_Borel} implies that $g_{\lambda}\circ \mathcal{A} = u_{\lambda}$ is Borel measurable.
\end{itemize}

Recall the definition of  $\mathcal{A}\colon X\to Z$ in the previous proof and recall that $\tilde{\mu} = \mathcal{A}\#\hat{m}|_{\mathcal{B}^Z}$.

\begin{lem}\label{lm:measure}
Let $\tilde{X}$ be a Polish space. Let $h\colon Z\to \tilde{X}$ be a Borel measurable function. Assume that $h \circ \mathcal{A} \colon X\to \tilde{X}$ is Borel measurable. Then
$h\# \tilde{\mu} = (h\circ \mathcal{A}) \# m$.
As a consequence, if $\tilde{X} = [c,+\infty)$ for some $c\in\R$, then
\begin{equation*}
\int_{Z} h \,  d \tilde{\mu} = \int_{X}h \circ \mathcal{A} \; d m.
\end{equation*}
\end{lem}

\begin{proof}
Let $B$ be any Borel set in $\tilde{X}$.
By the property of push-forward measure, $ h\#\tilde{\mu} (B) = \tilde{\mu} (h^{-1}(B))$. Since $h$ is Borel measurable, $h^{-1}(B) \in \mathcal{B}^{Z}$. Thus
$h\#\tilde{\mu} (B)  = {\mathcal{A}}\# \hat{m}(h^{-1}(B))$.
Next, by the property of the push-forward measure,
\begin{equation*}
{\mathcal{A}}\# \hat{m}(h^{-1}(B)) = \hat{m}(\mathcal{A}^{-1} h^{-1}(B)) = \hat{m}( (h \circ \mathcal{A})^{-1} (B)).
\end{equation*}
Since $h\circ \mathcal{A}$ is Borel measurable, we have that
$ (h \circ \mathcal{A})^{-1} (B) \in \mathcal{B}^X$. As a consequence,
\begin{equation*}
\hat{m}( (h\circ \mathcal{A})^{-1} (B)) = m( (h\circ \mathcal{A})^{-1} (B))  = (h\circ \mathcal{A})\# m (B).
\end{equation*}
Therefore, $h\#\tilde{\mu}(B) =  (h\circ \mathcal{A})\# m (B)$ for any Borel set $B\subseteq \tilde{X}$. This concludes the first part of the proof.
In the case where $X_3 = [c,+\infty)$ for some $c\in \R$, since
$ c = \int_{X_2} c\, d\tilde{\mu} = \int_{X_1} c \circ \mathcal{A} \, dm$, it suffices to prove the conclusion for $h-c$ in instead of $h$. Therefore, we can assume that $X_3 = \R_{+}$. By the change-of-variable formula for push-forward measures,
\begin{equation*}
\int_{Z} h d\tilde{\mu} =  \int_{\R_{+}} x\,  d (h\#\tilde{\mu} (x)).
\end{equation*}
Next, it follows from the equality $  h\# \tilde{\mu} = (h\circ \mathcal{A}) \# m$ that
\begin{equation*}
\int_{\R_{+}} x\,  d (h\#\tilde{\mu} (x)) = \int_{\R_{+}} x\,  d ((h\circ \mathcal{A})\# m  (x)).
\end{equation*}
Again, by the change-of-variable formula, we obtain that
\begin{equation*}
\int_{\R_{+}} x\,  d ((h\circ \mathcal{A})\# m  (x)) = \int_{ X } h\circ \mathcal{A} \; d m.
\end{equation*}
The conclusion follows.
\end{proof}

\begin{lem}\label{lm:lsc_Borel}
Under Assumption \ref{ass1}, the function $u_{\lambda}$ is upper semi-continuous for any $\lambda \in \mathcal{H}_f$, thus Borel measurable.
\end{lem}
\begin{proof}
Let $\lambda\in \mathcal{H}_f$. Since $g$ is bounded over $Z$, we have that $u_{\lambda}(x) > -\infty$ for any $x\in X$. Fix any $x\in X$. Let $y\in \br_{\lambda}(x)$. Let $(x_n\in X)_{n\geq 1}$ be a sequence converging to $x$. By the lower semi-continuity of $G_{\lambda}$, there exists $y_n \in Z_{x_n}$ such that $g_{\lambda}(x,y) = \lim_{n\to \infty} g_{\lambda}(x_n,y_n)$. Therefore,
\begin{equation*}
u_{\lambda}(x) = g_{\lambda} (x,y)= \lim_{n\to \infty} g_{\lambda}(x_n,y_n) \geq \limsup_{n\to \infty} u_{\lambda}(x_n).
\end{equation*}
We obtain the upper semi-continuuity of $u_{\lambda}$ for any $\lambda\in \mathcal{H}_f$. Since any upper semi-continuous function defined on a metric space is the limit of a monotonically decreasing sequence of continuous functions \cite[Thm.\@ 3]{tong1952}, we deduce that $u_{\lambda}$ is Borel measurable.
\end{proof}

\end{document}